\documentclass[11pt, reqno]{amsart}
\usepackage{hyperref}
\hypersetup{
	colorlinks=true,
	linkcolor=blue,
	filecolor=blue,
	urlcolor=red,
	citecolor=green,}

\usepackage{CJK}
\usepackage{subfigure}
\usepackage[graphicx]{realboxes}
\usepackage{tikz}
\usetikzlibrary {patterns}

\newtheorem{thm}{Theorem}[section]
\newtheorem{thm*}{Theorem A}
\newtheorem*{theorem*}{Theorem}
\newtheorem{cor}[thm]{Corollary}
\newtheorem{lem}[thm]{Lemma}

\newtheorem{rem}{Remark}

\theoremstyle{definition}
\newtheorem{defn}[thm]{Definition}
\theoremstyle{remark}

\numberwithin{equation}{section}

\usepackage{bm}

\newcommand{\thmref}[1]{Theorem~\ref{#1}}

\newcommand{\lemref}[1]{Lemma~\ref{#1}}

\newcommand{\corref}[1]{Corollory~\ref{#1}}

\newcommand{\R}{\mathbb{R}}

\newcommand{\Lie}{\mathcal{L}}

\newcommand{\ric}{{\rm Ric}}

\newcommand{\di}{{\rm div}}

\newcommand{\la}{\bm{\langle}}
\newcommand{\ra}{\bm{\rangle}}

\newcommand{\vol}{{\rm{vol}}}

\begin{document}
	
\title[Gradient Estimate and and Liouville Property]{Nash-Moser iteration approach to the logarithmic gradient estimates and Liouville Properties of quasilinear elliptic equations on manifolds}

\author{Jie He}
\address{School of Mathematics and Physics, Beijing University of Chemical Technology,  Chaoyang District, Beijing 100029, China}
	\email{hejie@amss.ac.cn}

\author{Jingchen Hu}
\address{Hua Loo-Keng Center for Mathematical Sciences, Academy of Mathematics and Systems Science, Chinese Academy of Sciences, Beijing 100190, China.}
\email{jingchenhu@amss.ac.cn}
	
\author{Youde Wang*}
\thanks{*Corresponding author}
\address{1. School of Mathematics and Information Sciences, Guangzhou University; 2. Hua Loo-Keng Key Laboratory
		of Mathematics, Institute of Mathematics, Academy of Mathematics and Systems Science, Chinese Academy
		of Sciences, Beijing 100190, China; 3. School of Mathematical Sciences, University of Chinese Academy of Sciences,
		Beijing 100049, China.}
\email{wyd@math.ac.cn}
	
\begin{abstract}
In this paper, we provide a new routine to employ the Nash-Moser iteration technique to analyze the local and global properties of positive solutions to the equation $$\Delta_pv + a|\nabla v|^qv^r =0$$ on a complete Riemannian manifold with Ricci curvature bounded from below, where $p>1$, $q$, $r$ and $a$ are some real constants. Assuming certain conditions on $a,\, p,\, q$ and $r$, we can derive universal and succinct Cheng-Yau type logarithmic gradient estimates for such solutions. In particular, we give the obvious expressions of constants in the logarithmic gradient estimate for entire solutions to the above equation (see \thmref{t10}). The gradient estimates enable us to obtain some Liouville-type theorems, Harnack inequalities and some local estimates near singularities for positive solutions. Some of our results are new even in the case the domain is an Euclidean space and $p=2$.
\end{abstract}
\maketitle
\leftline{\quad Key words: universal gradient estimate; Nash-Moser iteration; Liouville type theorem}
\leftline{\quad MSC 2020: 58J05; 35B45; 35J92}

\textbf{\tableofcontents}
	
\section{Introduction}
One trend in Riemannian geometry since the 1950's has been the study of how curvature affects global properties of partial differential equations and global quantities like the eigenvalues of the Laplacian (\cite{MR1333601,MR3275651}). It is well-known that gradient estimates are a fundamental and powerful tool in the geometry of manifolds (\cite{MR2518892}) and the analysis of partial differential equations on Riemannian manifolds such as the Liouville properties of solutions.

It is well-known that the Liouville theorem has had a huge impact across many fields, such as complex analysis, partial differential equations, geometry, probability, discrete mathematics and complex and algebraic geometry. The impact of the Liouville theorem has been even larger as the starting point of many further developments. For more details on the Liouville properties of harmonic functions and some related theory of function on a manifold we refer to an expository paper \cite{CM} written by T. H. Colding (also see \cite{CM1, CM2, CM3, MR1333601, MR431040}).

The most classical results about gradient estimates can be traced back to Cheng-Yau's gradient estimate for positive harmonic functions(see \cite{CY, MR431040}). Let $(M,g)$ be a complete non-compact Riemannian manifold $(M,g)$ with $\mathrm{Ric}_g\geq-(n-1)\kappa g$,  Cheng and Yau deduced for any positive harmonic function $v$ in a geodesic ball $B_R(o)\subset M$,
\begin{align}\label{equa0}
\sup_{B_{R/2}(o)}\frac{|\nabla v|}{v}\leq C_n\frac{1+\sqrt\kappa R}{R}.
\end{align}
An important feature of Cheng-Yau's estimate is that the RHS of \eqref{equa0} depends only on $n$, $\kappa$ and $R$, it does not depend on some other geometric quantities such as the injective radius and the Hessian or Laplacian of distance function etc. Cheng-Yau's estimate turned out to be very useful, Harnack inequality follows immediately from Cheng-Yau's estimate and Liouville theorem for global positive harmonic functions on non-compact manifolds with $\mathrm{Ric}_g\geq 0$ is also a direct consequence of \eqref{equa0}.

Usually, the above estimate is also called as universal a priori estimates. By using the word ``universal" here, we mean that our bounds are not only independent of any given solution under consideration but also do not require, or assume, any boundary conditions whatsoever. Actually, to derive universal a priori estimates for solutions of some partial differential equations is also one's principal purpose to explore local and global properties of these solutions (see \cite{MR615628, Dancer, MR1946918}).

Let $(M,g)$ be a complete non-compact Riemannian manifold with $\mathrm{Ric}_g\geq-(n-1)\kappa g$. We say an elliptic equation defined on $M$ satisfies Cheng-Yau type logarithmic gradient estimate if for any geodesic ball $B_R(o)\subset M$ and any (positive) solution $v$ of the equation on $B_R(o)$, there holds the above estimate \eqref{equa0}.

After Cheng-Yau's work, gradient estimates for many other equations defined on Riemannian manifolds are established and some corresponding Liouville theorems were derived as the consequences of gradient estimates (see for example \cite{han2023gradient,he2023gradient, MR0834612, MR2880214, MR4559367,MR431040, MR3866881}). This topic has attracted the attention of many mathematicians and there are a great deal of papers devoted to its study. But, there are only a little of papers on the universal boundedness of the solutions and their gradients to some degenerate quasilinear elliptic equation
$$\Delta_p v + f(v,\nabla v)=0 \quad\mbox{in}\,\, M$$
and corresponding parabolic equations on a general Riemannian manifold $(M, g)$ (see \cite{Ser3, CGS}).

\subsection{History and motivation}\

In fact, half century ago some inhomogeneous quasilinear elliptic equations written as
$$\Delta v + f(v,\nabla v)=0 \quad\mbox{in}\,\, \mathbb{R}^n$$
were considered by some mathematicians. For instance, Serrin \cite{Ser3} has ever proved the following beautiful and deep result:

{\em Let $v$ be an entire solution of the above equation. Suppose that $\partial f/\partial v\leq 0$ and that both $v$ and $\nabla v$ are bounded. Then $v$ must be constant.}

Still other Liouville theorems have been obtained for non-negative solutions of the Lane-Emden equation
$$\Delta v + v^r =0, \quad\mbox{in}\,\,\mathbb{R}^n$$
with $r>0$ (note that the previous result does not cover the above equation, since $v^r$ is increasing for $v>0$).
We would like to state another beautiful and deep result of Gidas and Spruck \cite{MR615628}:

{\em Assume $n>2$. Let $v$ be a non-negative solution of $\Delta v+v^r =0$ in an Euclidean space $\mathbb{R}^n $ with $1< r <(n+2)/(n-2)$ (critical Sobolev number). Then $v\equiv 0$.}

Subsequently, Serrin and Zou in \cite{MR1946918} considered the degenerate quasilinear elliptic equations
$$\Delta_p v + f(v)=0 \quad\mbox{in}\,\, \mathbb{R}^n$$
and some related differential inequalities on $\mathbb{R}^n$ and derived  a series of deep results such as some universal a priori estimates of solutions and Liouville theorems. We should mention that Serrin and Zou have slso observed that Liouville theorems can be seen as a consequence and a limiting case of universal boundedness theorems.

More recently, some new connections between Liouville-type theorems and local properties of nonnegative solutions to superlinear elliptic problems were studied, P. Pol\'a$\check{c}$ik, P. Quittner and P. Souplet in \cite{PQS} developed a general method for derivation of universal, pointwise, a priori estimates of local solutions from Liouville-type theorems, which provides a simpler and unified treatment for such questions. The method is based on rescaling arguments combined with a key “doubling” property, and it is different from the classical rescaling method of Gidas and Spruck \cite{MR615628}. As an important heuristic consequence of their approach, it turns out that universal boundedness theorems for local solutions and Liouville-type theorems are essentially equivalent.

Usually, one is interested in several global questions concerning entire solutions in $\mathbb{R}^n$:

(I). Liouville type properties;

(II). Asymptotic behavior and uniqueness;

(III). Symmetry properties.

A natural motivation is how to extend the above considerations (I) and (II) to degenerate quasilinear elliptic equations of the same form, which are defined on a complete Riemannian manifold, since a general manifold is not of abundant symmetry. In particular, one wants to know whether or not Cheng-Yau type gradient estimate holds true with an universal bound.
	
For simplicity, in this article we are concerned with the equation
\begin{equation}\label{equa1}
\begin{cases}
\Delta_pv+a|\nabla v|^qv^r =0, &\quad \text{in}\,\, M;\\
v>0, &\quad \text{in}\,\, M;\\
p>1,~ a, ~q, ~r\in\R, &
\end{cases}
\end{equation}
where $(M,g)$ is a complete Riemannian manifold with Ricci curvature bounded from below. Equation \eqref{equa1} arises from many classical equations
and there are many questions related to equation (\ref{equa1}).
	
The primary objectives of the present paper are twofold: Firstly, we want to establish an exact Cheng-Yau type gradient estimates for positive solutions to equation \eqref{equa1} on a Riemannian manifold with Ricci curvature bounded below. Secondly, as applications we would like to show some global properties of solutions to \eqref{equa1}, for instance, Liouville-type theorems. In addition, we try to extend the ranges of values of $p,\, q$ and $r$, obtained in the previous work, such that the Liouville results on solutions to equation \eqref{equa1} on an Euclidean space hold true. To achieve these goals, we need to develop an adaptive pointwise estimate for the linearized operator of the $p$-Laplace operator and then to provide a new routine to use directly the Nash-Moser iteration method.
	
Now, let's recall some relevant work with the above equation \eqref{equa1}. In the case $a=0$, \eqref{equa1} reduces to the $p$-Laplace equation.  In 2011, Wang and Zhang \cite{MR2880214} proved any positive $p$-harmonic function $v$ on  $(M,g)$ with Ricci curvature $\mathrm{Ric}_g\geq-(n-1)\kappa g$ satisfies
\begin{align*}
\sup_{B_{R/2}(o)}\frac{|\nabla v|}{v}\leq C_{n,p}\frac{1+\sqrt\kappa R}{R}.
\end{align*}
Later on, Sung and Wang \cite{MR3275651} in 2014 proved that the optimal constant $C_{n,p}$ in the above estimate is $(n-1)/(p-1)$. Wang and Zhang extended Cheng-Yau's result for the case $p=2$ (\cite{MR431040}) to any $p>1$ and improved Kotschwar-Ni's results \cite{MR2518892} by weakening sectional curvature condition to the Ricci curvature condition.

Next, we recall some known results for this equation with $a\neq0$, which are closely related to our present situation. Of course, due to the large number of papers dealing with this topic and its generalizations to other problems involving quasilinear elliptic operators, it is not possible to produce here an exhaustive bibliography.
\medskip

{\bf Case (1): $a=1$ and $q=0$}	

In the case $a=1$, $q=0$ and $M=\R^n$, the equation (\ref{equa1}) reduces to the classical Lane-Emden-Fowler equation
\begin{align}\label{equa:1.2}
\begin{cases}
\Delta_pv + v^r=0, \quad \text{in}\,\, \R^n;\\
v>0, \quad \text{in}\,\, \R^n,
\end{cases}
\end{align}
where $r>0$ \text{and}  $1<p<n$. The equation (\ref{equa:1.2}) has already been the subject of countless publications (see \cite{MR1004713, MR1134481, MR982351, MR1121147, MR615628, MR829846, PQS, MR1946918, Ser1, Ser2, Ser3, Sou}). One of the questions solved (see \cite{MR1946918}) is that, if $p<n$, the Liouville property holds (i.e. equation (\ref{equa:1.2}) admits no positive solution) if and only if
	$$
	r\in \left(0,\,\, p^*-1 \right),
	$$
where $p^*=np/(n-p)$ is the Sobolev duality of $p$.

Recently, one also discussed the gradient estimates of positive solutions to Lane-Emden-Fowler equation on a complete manifold $(M,g)$ with $\mathrm{Ric}_g\geq -(n-1)\kappa$, i.e.
\begin{align}\label{equa:1.3}
\begin{cases}
\Delta_pv + av^r=0, \quad &\text{in}\,\, M;\\
v>0,\quad &\text{in}\,\, M.
\end{cases}
\end{align}
In the case $p=2$ and $a>0$, Wang and Wei \cite{MR4559367} derived Cheng-Yau type gradient estimates for positive solutions to \eqref{equa:1.3} under the assumption
	$$
	r\in \left(-\infty,\quad \frac{n+1}{n-1}+\frac{2}{\sqrt{n(n-1)}}\right).
	$$
For any $p>1$ and $a>0$, He, Wang and Wei \cite{he2023gradient} proved that  the Cheng-Yau type gradient estimate holds for positive solutions of \eqref{equa:1.3} when
	$$
	r\in \left(-\infty,\quad \frac{n+3}{n-1}(p-1)\right).
	$$
We refer to \cite{MR3912761, MR3866881} when the $p$-Laplace operator in the equation \eqref{equa:1.3} is replaced by weighted $p$-Laplace Lichnerowicz operator.
\medskip	


{\bf Case (2): $a=-1$ and $r=0$}
		
On the other hand, if $r=0$ and $a= -1$, then (\ref{equa1}) reduces to the so called quasilinear Hamilton-Jacobi equation
\begin{align}\label{hamjoco}
\Delta_p v - |\nabla v|^q = 0.
\end{align}
Lions \cite{MR833413} proved that any $C^2$ solution to (\ref{hamjoco}) on $\R^n$ with $q>1$ and $p=2$ must be a constant. Later,  Bidaut-V\'eron,  Garcia-Huidobro and V\'eron \cite{MR3261111} proved that when $n\geq p>1$ and $q>p-1$, the gradient of the solutions to (\ref{hamjoco}) on $\Omega\subset\R^n$ can be controlled by the distance fuction to the boundary of $\Omega$
\begin{align}\label{jfathma}
|\nabla v(x)|\leq c(n,p,q)(d(x,\partial\Omega))^{-\frac{1}{q+1-p}},\quad \forall x\in\Omega\subset\mathbb{R}^n,
\end{align}
and obtained some Liouville-type theorems. They also extended their results to Riemannian manifolds with sectional curvature bounded from below. Recently, Han, He and Wang in \cite{han2023gradient} proved any solution $v\in C^1(B_R(o))$ of equation \eqref{hamjoco} with $q>p-1$ defined on a geodesic ball $B_R(o)$ of a complete Riemannian manifold $(M,g)$ with $\mathrm{Ric}_g\geq-(n-1)\kappa g$ satisfies
\begin{align}\label{hanhewang}
\sup_{B_{R/2}(o)} |\nabla v|  \leq C(n,p,q)\left(\frac{1+\sqrt{\kappa}R}{R}\right)^{q-p+1}.
\end{align}

{\bf Case (3): $a\neq0$, $q\neq0$ and $r\neq0$}
		
If $a$, $q$ and $r$ are all non-zero, one also discussed a more general case where $p>1$, i.e., the equation
\begin{align}\label{equa6}
\begin{cases}
\Delta_p  v + |\nabla v|^q v^r = 0, \quad &\text{in}\,\, \R^n;\\
v>0,  &\text{in }\, \R^n.
\end{cases}
\end{align}
Provided $n>p>1$, Mitidieri and Pohozaev (Theorem 15.1 in \cite{MR1879326}) proved that if $r>0$, $q\geq0$ and $q+r>p-1$ such that
\begin{align}
r(n-p)+q(n-1)<n(p-1)
\end{align}
holds, then all nonnegative functions $u \in C^1(\mathbb{R}^n)$ satisfying
$$-\Delta_p v \geq |\nabla v|^qv^r\geq 0$$
are constant.
	
Recently, Chang, Hu and Zhang \cite{MR4400071} proved that if $n\geq 2$,  $\,p,\,q$ and $r$ satisfy one of the following two conditions,
\begin{itemize}
\item $0\leq q<p,\, r>0\quad $ and $\quad 1<\frac{q}{p-1}+\frac{r}{p-1}<\frac{n+4}{n}$;
		
\item $1<p<n, \, r\geq 0, \, q\geq0\quad $ and $\quad r(n-p) + q(n-1) < n(p - 1)$,
\end{itemize}
then positive solutions to the equation (\ref{equa6}) are constants, and in some sense their results can be viewed as a complement of Theorem 15.1 in \cite{MR1879326}. Bidaut-V\'eron \cite{MR4234084} obtained some Liouville-type results for (\ref{equa6}) if $n >p> 1,\, r \geq 0$ and $q \geq p$ without the assumption of boundedness on the solution. For more details on \eqref{equa1} on $\R^n$, we refer to \cite{MR2509378, MR2737854, MR3072673, MR1879326} and the references therein.
	
In particular, as $p=2$, in many cases one would like to assume that the superlinearity of the right-hand side (i.e., $q+r-1>0$ and $0 \leq q < 2$) of the equation \eqref{equa1} on an Euclidean space. For more details, we refer to \cite{MR3959864} (or see Subsection 5.1).
	
Besides, they also gave some other regions of $(q,r)$ for which the Liouville property holds in the case $0\leq q<2$, $r\geq 0$ and $q+r>1$.
	
Filippucci, Pucci and Souplet \cite{MR4095468} also proved that any positive bounded classical solution of (\ref{equa1.5}) is identically equal to a constant provided $q \geq 2$ and $r > 0$.
	
As for the equation (\ref{equa1}) defined on a Riemannian manifold, to our knowledge, the only known Liouville type result is due to Sun, Xiao and Xu \cite{MR4498474}. They showed that, under some additional conditions on the volume growth of geodesic ball on $M$, there exists no nontrivial non-negative solution of equation \eqref{equa1} with $a>0$. In fact, they discussed the nonexistence of non-negative solutions to the following differential inequality $\Delta_pv + |\nabla v|^qv^r \leq 0$ in $M$.

\subsection{Main results and some remarks}\

Now we are in the position to state our main results. Firstly, we can establish a Liouville theorem for equation \eqref{equa1} on a non-compact complete Riemannian manifold with non-negative Ricci curvature on the same range of values of $p$, $q$ and $r$ where \cite[Theorem 15.1]{MR1879326} (Liouville theorem) holds (also see \cite{MR4400071}).
	
\begin{thm}\label{t4}
Let $(M,g)$ be complete non-compact manifold with non-negative Ricci curvature. Suppose that $1<p<n$ and $r>0,\, q\geq0,\, q+r>p-1$, and $a>0$. If
\begin{align*}
r+\frac{n-1}{n-p}q<\frac{n(p-1)}{n-p},
\end{align*}
then the problem
\begin{align}\label{veron}
\begin{cases}
\Delta_pv +av^r|\nabla v|^q\leq 0, &\quad\text{in}\,\, M;\\
v > 0, & \quad\text{in}\,\, M
\end{cases}
\end{align}
admits no solution $v\in C^1(M)$.
\end{thm}

The above theorem required the restrictions: $q\geq 0$ and $r>0$. Naturally, one would like to ask whether the restrictions: $q\geq 0$ and $r>0$ can be discarded in the above Liouville theorem or not? For this end we need to estimate the gradients of positive solutions. If $p>1$, we obtain the following concise and universal gradient estimate, which can cover or extend some previous results, and give some new conclusions for some related equations:

\begin{thm}\label{thm1}
Let $(M,g)$ be an $n$-dimensional complete Riemannian manifold with Ricci curvature satisfying $\ric_g\geq -(n-1)\kappa g$, where $\kappa$ is a nonnegative constant. Suppose that $p>1$ and $v$ is a $C^1$-positive solution of \eqref{equa1} on a geodesic ball $B_R(o)\subset M$. If $n$, $p$, $q$, $r$ and $a$ satisfy
\begin{align}\label{cond1a}
a\left(\frac{n+1}{n-1}-\frac{q+r}{p-1}\right)\geq 0,
\end{align}
or
\begin{align}\label{cond2a}
1<\frac{q+r}{p-1}<\frac{n+3}{n-1},\quad  \forall a\in \R,
\end{align}
then there holds true
\begin{align}\label{equa:1.8}
\sup_{B_{R/2}(o)} \frac{|\nabla v|}{v}\leq C(n,p,q,r)\frac{1+\sqrt{\kappa}R}{R},
\end{align}
where the constant $C(n,p,q,r)$ depends only on $n$, $p$, $q$ and $r$.
\end{thm}
	
\begin{rem}\label{cor1}
Let $(M,g)$ be an $n$-dimensional complete Riemannian manifold with Ricci curvature $\ric_g\geq -(n-1)\kappa g$, where $\kappa$ is a non-negative constant.
		\begin{enumerate}
			\item In the case where $a =0$, Wang-Zhang's gradient estimate (see \cite{MR2880214}) can be recovered by the above \thmref{thm1}.
			
			\item In the case $q=0$, suppose that $n$, $p$, $r$ and $a$ satisfy
			\begin{align}\label{cond1}
				a\left(\frac{n+1}{n-1}-\frac{ r}{p-1}\right)\geq 0,
			\end{align}
			or
			\begin{align}\label{cond2}
				1<\frac{ r}{p-1}<\frac{n+3}{n-1},\quad  \forall a\in \R.
			\end{align}
Then, the above estimate (\ref{equa:1.8}) can cover the gradient estimate obtained in \cite{he2023gradient}.
\item In the case $r=0$ and $a=-1$, \thmref{thm1} implies that any positive solution $v$ of the quasi-linear Hamilton-Jacobi equation \eqref{hamjoco} defined on a geodesic ball $B_R(o)\subset M$ satisfies \eqref{equa:1.8} if $q>p-1$.
			
The results due to Han, He and Wang \cite{han2023gradient} demonstrate that any solution $u$ of the quasi-linear Hamilton-Jacobi equation defined on a geodesic ball $B_R(o)\subset M$ satisfies \eqref{hanhewang}. For positive solution $v$ of \eqref{equa1}, $\ln v$ satisfies a different equation (see \eqref{equa2.1}), so the corresponding gradient estimates for $\ln v$ and $v$ are different from each other. However, the necessary condition for both gradient estimates are the same, i.e., $q>p-1$.
\end{enumerate}
\end{rem}
	
By carefully analyzing the conditions \eqref{cond1a} and \eqref{cond2a}  in \thmref{thm1}, it is easy to see that the following results holds.
\begin{cor}\label{cor1.3}
Let $p>1$ and $(M,g)$ be an $n$-dim complete Riemannian manifold with $\ric_g\geq -(n-1)\kappa g$, where $\kappa$ is a non-negative constant. If
		\begin{align}\label{cond3}
			a\geq0 \quad\text{ and }\quad q+r <\frac{n+3}{n-1}(p-1),
		\end{align}
or
\begin{align}\label{cond4}
a\leq 0 \quad\text{ and }\quad q+r>p-1,
\end{align}
then, for any $C^1$-positive solution $v$ of (\ref{equa1}) on a geodesic ball $B_R(o)\subset M$, we have
\begin{align*}
\sup_{B_{R/2}(o)} \frac{|\nabla v|}{v}\leq C(n,p,q,r)\frac{1+\sqrt{\kappa}R}{R}.
\end{align*}
\end{cor}
	
\begin{rem}
Firstly \thmref{thm1} is the Cheng-Yau type estimate for positive solutions of equation \eqref{equa1}. In the case $p=2$ and $a=1$, the assumption \eqref{cond3} is just \eqref{vercondi2}. So, we drop the restriction $q\geq 0$ supposed in Theorem B of \cite{MR3959864} but our results can not cover completely the theorem.
\end{rem}
	
By the gradient estimations obtained in the above, we can get Liouville-type results for such solutions on a Riemannian manifold with non-negative Ricci curvature.
	
\begin{thm}\label{thm1.4}
Let $p>1$ and $(M,g)$ be an $n$-dimensional complete non-compact Riemannian manifold with non-negative Ricci curvature. If
		\begin{align*}
			a\geq0 \quad\text{ and }\quad q+r <\frac{n+3}{n-1}(p-1),
		\end{align*}
or
		\begin{align*}
			a\leq 0 \quad\text{ and }\quad q+r>p-1,
		\end{align*}		
then any $C^1$ smooth positive solution of equation \eqref{equa1} on $M$ is a constant.
\end{thm}
	
\begin{rem}
To our best knowledge,  \thmref{thm1.4} is the first Liouville type result for  equation (\ref{equa1}) on Riemannian manifolds with non-negative Ricci curvature. Sun-Xiao-Xu \cite{MR4498474} also obtained some Liouville-type results for (\ref{equa1}) under the assumption on volume growth rate of geodesic ball. The classical Cheng-Yau's results on harmonic functions on complete Riemannian manifolds show that the condition on the Ricci tensor in \thmref{thm1.4} is natural.
\end{rem}

\begin{rem}
When $a=1$ and the equation (\ref{equa1}) is defined on $\R^n$, the previous results (see \cite{MR4234084, MR3959864, MR4400071, MR4095468}) are derived by the Bernstein method. The essential difference between the Liouville-type theorems of this paper and the previous Liouville-type theorems is that $q$ and $r$ may be negative in the present situation, however, one needs to assume that $q$ and $r$ are nonnegative in the previous results. Moreover, we also consider the case where $a<0$ in \thmref{thm1}.
\end{rem}
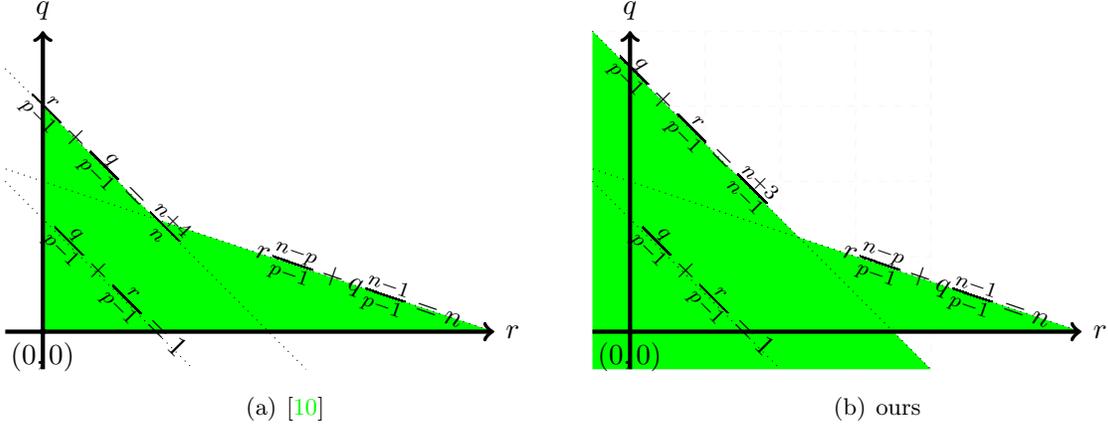
\begin{figure}[ht]\label{fig1}
\centering
		\subfigure[\cite{MR4400071}]{
			\begin{minipage}[t]{0.48\linewidth}
				\begin{tikzpicture}
					\path[fill=green](0, 0)--(6,0)--(0, 2);
					\path[fill=green](0, 3)--(0, 1.5)--(1.5, 0)--(9/4, 0)--(9/4, 3/4);
				\draw[dotted] (-0.5, 3.5) node[anchor=north west]{\rotatebox{315}{$ \frac{r}{p-1} +  \frac{q}{p-1} = \frac{n+4}{n} $} } -- (3.5, -0.5);
				\draw[dotted]
				(-0.5, 2)--(-0.2, 1.7) node[anchor=north west]{\rotatebox{315}{$\frac{q}{p-1}+\frac{r}{p-1}=1$}} --(2, -0.5);
				\draw[dotted] (6, 0) -- (-0.5, 13/6);
					\filldraw[black] (0,0) circle (1pt) node[anchor=north]{$(0,0)$};
					\draw (4.2,0.6) node[] {\rotatebox{341}{$r\frac{n-p}{p-1} + q\frac{n-1}{p-1} = n $}};
					\draw[->,ultra thick] (-0.5,0)--(6,0) node[right]{$r$};
					\draw[->,ultra thick] (0,-0.5)--(0,4) node[above]{$q$};
				\end{tikzpicture}			
			\end{minipage}%
		}%
\subfigure[ours]{
\begin{minipage}[t]{0.48\linewidth}
\begin{tikzpicture}
					\draw[help lines, color=red!5, dashed] (-0.5,-0.5) grid (4,4);
					\path[fill=green](-0.5, 4)--(-0.5,-0.5)--(4, -0.5);
					\path[fill=green](0, 0)--(6,0)--(0, 2);
					\draw[dotted]
					(-0.5, 4)node[anchor=north west]{\rotatebox{315}{$\frac{q}{p-1}+\frac{r}{p-1}=\frac{n+3}{n-1}$}} --(4, -0.5);
					\draw[dotted]
					(-0.5, 2)--(-0.2, 1.7) node[anchor=north west]{\rotatebox{315}{$\frac{q}{p-1}+\frac{r}{p-1}=1$}} --(2, -0.5);
						\draw[dotted] (6, 0) -- (-0.5, 13/6);
					\filldraw[black] (0,0) circle (1pt) node[anchor=north]{$(0,0)$};
					\draw (4.2,0.6) node[] {\rotatebox{341}{$r\frac{n-p}{p-1} + q\frac{n-1}{p-1} = n $}};
					\draw[->,ultra thick] (-0.5,0)--(6,0) node[right]{$r$};
					\draw[->,ultra thick] (0,-0.5)--(0,4) node[above]{$q$};
\end{tikzpicture}			
\end{minipage}
		}%
\centering
\caption{ The region of $(q,r)$ where Liouville result holds when $n=4$ and $p=3$.}
\end{figure}
	
\begin{rem}
In the case where $a=1$, $q\geq 0$ and $r\geq 0$, even for the equation (\ref{equa1}) which is defined on $\R^n$, the results obtained here are new. For the Liouville-type results established in the present paper, there is no upper bound restriction on $p$ while one always need to assume that $p<n$ in the previous all results. In the case $p=2$ and $a>0$, the range of $(q,r)$ in \thmref{thm1.4} is just one of the conditions given in \cite{MR3959864}, we extend one of the Liouville results proved for the case $p=2$ in \cite{MR3959864} to the case $p>1$. Our results (\thmref{thm1.4} \thmref{t4}) and the Liouville type results in  \cite{MR4234084} can not cover each other, but our results can completely cover the Liouville results in \cite{MR4400071, MR4095468}. In the case $n=4$ and $p=3$, the above Figure $1$ shows the regions of $(q,r)$ where the Liouville result holds by different methods.
\end{rem}
	
\begin{thm}\label{thm1.5}
Let $(M,g)$ be a complete non-compact Riemannian manifold with Ricci curvature $\ric_g\geq-(n-1)\kappa g$, where $\kappa$ is a non-negative constant. Suppose $v\in C^1(B_R(o))$ is a positive solution to equation \eqref{equa1}, defined on a geodesic ball $B_R(o)\subset M$, with $a$, $p$, $q$ and $r$ satisfying \eqref{cond3} or \eqref{cond4}. Then, for any $x, y\in B_{R/2}(o)$ there holds
$$
v(x)/v(y)\leq  e^{C(n,p,q,r)(1+\sqrt{\kappa}R)}.
$$
Moreover, if $v\in C^1(M)$ is an entire positive solution to equation \eqref{equa1} with $a$, $p$ and $q$ satisfying \eqref{cond3} or \eqref{cond4} in $M$, then, for any $x,\, y\in M$ there holds
$$
v(x)/v(y)\leq  e^{C(n,p,q,r)\sqrt{\kappa}d(x,y)},
$$
where $d(x,y)$ is the geodesic distance between $x$ and $y$.
\end{thm}

\begin{rem}
When $a=1$ and the equation (\ref{equa1}) is defined on $\R^n$, we can recover Theorem A in \cite{MR3959864} by making use of the above Harnack inequality and the arguments in \cite{BL}, for the details we refer to P.1493 of \cite{MR3959864}.
\end{rem}

\begin{thm}\label{thm1.6}
Let $(M,g)$ be a complete non-compact Riemannian manifold with Ricci curvature $\ric_g\geq-(n-1)\kappa g$, where $\kappa$ is a non-negative constant. Suppose $v\in C^1(B_R(o)\setminus\{0\})$ is a positive solution to equation \eqref{equa1} with $a$, $p$ and $q$ satisfying \eqref{cond3} or \eqref{cond4}.

1. If $v$ satisfies in $B_R(o)\setminus\{0\}$ that

\begin{equation}\label{tt}
v(x)\leq cd(x,0)^{-\lambda}
\end{equation}
for some constant $c$ and some exponent $\lambda > 0$, then there exists $c_1=c_1(n,p,q,r,c)$ such that
$$|\nabla v(x)|\leq c_1(1+\kappa d(x, 0))d(x, o)^{-(\lambda+1)},\quad \forall x\in B_{R/2}(o).$$

2. If $v$ satisfies \eqref{tt} in $M\setminus B_R(o)$ for some constant $c$ and some exponent $\lambda > 0$, then there exists $c_1=c_1(n,p,q,r,c)$ such that
$$|\nabla v(x)|\leq c_1(1+\kappa d(x, 0))d(x, o)^{-(\lambda+1)}, \quad \forall x\in M\setminus B_{2R}(o).$$
\end{thm}

\subsection{Applications to singularity analysis in $\mathbb R^n$.}\

Now we consider the case of equation \eqref{equa1} defined on a domain of $\Omega\subset\mathbb R^n$. Assume that \eqref{cond3} or \eqref{cond4} is satisfied, and $v$ is of possible singularities. Firstly, we can obtain the following asymptotic estimate near a singular point.

\begin{thm}\label{t7}
(i). Let $\Omega\subset\mathbb{R}^n$ be a domain containing $0$ and $p=2$. Assume that $a$, $q$ and $r$ fulfill \eqref{cond3}, i.e.,
$$a\geq 0 \quad\text{ and }\quad q+r <\frac{n+3}{n-1}.$$
If $v\in C^2(\Omega\setminus\{0\})$ is a positive solution of \eqref{equa1} in $\Omega\setminus\{0\}$, then, in the case $n\geq3$ the following estimate holds in a neighborhood of $0$
\begin{equation*}\label{Ver}
v(x)+|x||\nabla v(x)| \leq c|x|^{2-n},
\end{equation*}
where $c$ is a constant depending on $u$; and in the case $n=2$ there holds in a neighborhood of $0$
\begin{equation*}\label{Ver}
v(x)+|x||\nabla v(x)| \leq c\log\frac{1}{|x|},
\end{equation*}
where $c$ is a constant depending on $v$.

(ii). Let $\Omega\subset\mathbb{R}^n$ be a domain containing $0$ and $p=2$. Assume that $a$, $q$ and $r$ fulfill \eqref{cond4}, i.e.,
$$a\leq 0 \quad\text{ and }\quad q+r>p-1.$$
If $v\in C^2(\Omega\setminus\{0\})$ is a positive solution of \eqref{equa1} in $\Omega\setminus\{0\}$, then, in the case $n\geq3$ there holds in a neighborhood of $0$ with a constant $c$ depending on $v$
\begin{equation*}\label{Ver}
v(x) \geq c|x|^{2-n};
\end{equation*}
and in the case $n=2$ there holds in a neighborhood of $0$ with a constant $c$ depending on $v$
\begin{equation*}\label{Ver}
v(x) \geq c\log\frac{1}{|x|}.
\end{equation*}
\end{thm}

\begin{rem}
It is not difficult to find that the conclusion (i) in the above theorem covers the conclusion stated in Theorem A in \cite{MR3959864} (see Subsection 5.1 of this article).
\end{rem}

Secondly, we can give the bound of the gradient of logarithm of the positive solution at any point $x\in\Omega$ by its distance to the boundary.

\begin{cor}\label{tui1}
If $v$ is a positive solution of \eqref{equa0}, then for any $x\in\Omega$ there holds
\begin{align}
|\nabla\ln v(x)|\leq C(n,p,q,r)d(x, \partial\Omega)^{-1}.
\end{align}
\end{cor}

By the above corollary, we can estimate the order of singularity of the solution in a punctured domain.
\begin{cor}\label{tui2}
Let $\Omega$ be a domain containing $0$ and $R>0$ be a positive number such that $d(0,\partial\Omega) \geq2R$. Assume that $v\in C^2(\Omega\setminus\{0\})$ is a positive solution to equation \eqref{equa0}. Then, for any $x\in B_{R}\setminus \{0\}$ we have
\begin{align}
v(x)\leq \sup_{|y|=R} v(y)\cdot \left(\frac{R}{|x|}\right)^{C(n,p,q,r)}.
\end{align}
\end{cor}

\begin{cor}\label{tui3}
Let $\Omega$ be a bounded domain with a $C^2$ boundary. Assume that $v\in C^2(\Omega\setminus\{0\})$ is a positive solution to equation \eqref{equa0}. Then, there exists $\delta^*>0$ such that, for any $x\in \Omega_{\delta^*}$,  where $\Omega_{\delta^*}:=\{z\in\Omega:d(z, \partial\Omega)\leq\delta^*\}\subset\Omega$, the following is true
\begin{align*}
v(x)\leq \sup_{d(z^*,\, \partial\Omega)=\delta^*}  v(z^*)\left(\frac{\delta^*}{d(x,\partial\Omega)} \right)^{C(n,p,q,r)}.
\end{align*}
\end{cor}

\subsection{Global gradient estimates}\

If $v$ is an entire positive solution to \eqref{equa1}, by the local gradient estimate, we can give the global bound of the gradient of positive solutions. The following theorem gives an explicit expression of global bound for gradient of positive solution to \eqref{equa1}.
\begin{thm}\label{t10}
Let $(M,g)$ be a complete non-compact Riemannian manifold with $\mathrm{Ric}_g\geq-(n-1)\kappa$. Assume that $v$ is an entire positive solution to \eqref{equa1} in $M$. Then,
\begin{enumerate}
\item if
\begin{align*}
a\left(\frac{n+1}{n-1}-\frac{q+r}{p-1}\right)\geq 0,
\end{align*}
there holds
\begin{align*}
\frac{|\nabla v|}{v}(x)\leq \frac{(n-1)\sqrt{\kappa}}{p-1}, \quad\forall x\in M;
\end{align*}
\item if
\begin{align*}
1<\frac{q+r}{p-1}<\frac{n+3}{n-1},\quad  \forall a\in \R,
\end{align*}
there holds
\begin{align*}
\frac{|\nabla v|}{v}(x)\leq \frac{2\sqrt{\kappa}}{(p-1)\sqrt{\left(\frac{q+r}{p-1}-1\right)\left(\frac{n+3}{n-1}-\frac{q+r}{p-1}\right)}}, \quad\forall x\in M.
\end{align*}
\end{enumerate}
\end{thm}

The remainder of our paper is organized as follows. In Section 2, we will give a meticulous estimate of $\mathcal{L} \left(|\nabla \log v|^{2\alpha}\right)$ (see \eqref{linearization} for the explicit definition of the operator $\mathcal{L}$) and recall Saloff-Coste's Sobolev embedding theorem. The proof of Theorem \ref{t4} is provided in Section 3. In section 4, we use delicately the Nash-Moser iteration to provide the proofs of the main results in this paper. We give some applications of main results in the case $M$ is a domain contained in an Euclidean space in Section 5. In final section the proof for global gradient estimates are provided.

\section{preliminary}
	
In this paper, $(M,g)$ is an $n$-dimensional Riemann manifold and $\nabla$ is the associated Levi-Civita connection. For any function $\varphi\in C^1(M)$, we denote $\nabla \varphi\in \Gamma(T^*M)$ by $\nabla \varphi(X)=\nabla_X\varphi$. The volume form is denoted as
$$\vol=\sqrt{\det(g_{ij})}d x_1\wedge\ldots\wedge dx_n$$
where $(x_1,\ldots, x_n)$ is a local coordinate. For simplicity, we omit the volume form  in  integrations.
	
For $p>1$, the $p$-Laplace operator is defined by
	$$
	\Delta_pu:=\di\left(|\nabla u|^{p-2}\nabla u\right).
	$$
The solution of $p$-Laplace equation $\Delta_pu=0$ is the critical point of the energy functional
	$$
	E(u)=\int_M|\nabla u|^p.
	$$
	
\begin{defn}\label{def1}
A function $v$ is said to be a positive (weak) solution of equation \eqref{equa1} on a region $\Omega$ if $v\in C^1(\Omega), v>0 $ and for all $\psi\in C_0^\infty(\Omega)$, we have
\begin{align*}
-\int_\Omega|\nabla v|^{p-2}\la\nabla v,\nabla\psi\ra+a\int_\Omega|\nabla u|^qv^r\psi =0.
\end{align*}
\end{defn}
It is worth mentioning that any solution $v$ of equation (\ref{equa1}) satisfies $v\in W^{2,2}_{loc}(\Omega)$ and $v\in C^{1,\alpha}(\Omega)$ for some $\alpha\in(0,1)$(for example, see \cite{MR0709038, MR0727034,MR0474389}). Moreover, $v$ is in fact smooth away from $\{\nabla v=0\}$.
	
Next, we recall Saloff-Coste's Sobolev inequalities (see \cite[Theorem 3.1]{saloff1992uniformly}) which shall play an important role in our proof of the main theorem.
	
\begin{lem}[\cite{saloff1992uniformly}]\label{salof}
Let $(M,g)$ be a complete manifold with $\mathrm{Ric}_g\geq-(n-1)\kappa$. For $n>2$, there exists a positive constant $c_0$ depending only on $n$, such that for all $B\subset M$ of radius R and volume $V$ we have for $f\in C^{\infty}_0(B)$
		$$
		\|f\|_{L^{\frac{2n}{n-2}}(B)}^2\leq e^{c_0(1+\sqrt{\kappa}R)}V^{-\frac{2}{n}}R^2\left(\int_B|\nabla f|^2+R^{-2}f^2\right).
		$$
For $n=2$, the above inequality holds if we replace $n$ by some fixed $n'>2$.
\end{lem}
	
For any positive solution $v$ to the equation (\ref{equa1}), if we take a logarithmic transformation $v = e^{-\frac{u}{p-1}}$, then we have
\begin{align}\label{equa2.1}
\Delta_pu-|\nabla u|^p-be^{cu}|\nabla u|^q=0,
\end{align}
where
\begin{align}\label{defofbc}
b =a (p-1)^{p-1-q}\quad\text{and} \quad c=1-\frac{q+r}{p-1}.
\end{align}
	
Denote $f=|\nabla u|^2$. It is easy to see that (\ref{equa2.1}) becomes
	\begin{align}\label{equa2.3}
		\Delta_pu-f^\frac{p}{2}-be^{cu}f^\frac{q}{2}=0.
	\end{align}
Now we consider the linearization operator $\Lie$ of $p$-Laplace operator:
	\begin{align}\label{linearization}
		\Lie(\psi)=\di\left(f^{p/2-1}A\left(\nabla \psi\right)\right),
	\end{align}
where
	\begin{align}\label{defoff}
		f = |\nabla u|^2
	\end{align}
and
	\begin{align}\label{defofA}
		A(\nabla\psi) = \nabla\psi+(p-2)f^{-1}\la\nabla \psi,\nabla u\ra\nabla u.
	\end{align}
We first derive a useful expression of $\mathcal L(f^\alpha)$ for any $\alpha>0$.
	
\begin{lem}\label{lem2.3}
For any $\alpha >0$, the following identity holds point-wisely in $\{x : f(x) > 0\}$,
		\begin{align}\label{bochner1}
			\begin{split}
				\mathcal{L} (f^{\alpha}) =
				&\alpha(\alpha+\frac{p}{2}-2)f^{\alpha+\frac{p}{2}-3}|\nabla f|^2 + 2\alpha f^{\alpha+\frac{p}{2}-2} \left(|\nabla\nabla u|^2
				+ \ric(\nabla u,\nabla u) \right)\\
				&+\alpha(p-2)(\alpha-1)f^{\alpha+\frac{p}{2}-4}\langle\nabla f,\nabla u\rangle^2 + 2\alpha f^{\alpha-1}\langle\nabla\Delta_p u,\nabla u\rangle.
			\end{split}
		\end{align}
	\end{lem}
	
\begin{proof}
By the definition of $A$ in (\ref{defofA}), we have
$$A\Big(\nabla (f^{\alpha})\Big) = \alpha f^{\alpha-1}\nabla f + \alpha(p-2)f^{\alpha-2}\langle\nabla f,\nabla u\rangle \nabla u = \alpha f^{\alpha-1}A(\nabla f).$$
Hence
\begin{align*}
\mathcal{L} (f^{\alpha}) &= \alpha \text{div}\Big( f^{\alpha-1}f^{\frac{p}{2}-1}A(\nabla f)\Big) = \alpha \Big\langle\nabla(f^{\alpha-1}),f^{\frac{p}{2}-1}A(\nabla f)\Big\rangle + \alpha f^{\alpha-1}\mathcal{L} (f).
\end{align*}
		
Direction computation shows that
\begin{align}\label{equ2.6}
\alpha \Big\langle\nabla( f^{\alpha-1}),f^{\frac{p}{2}-1}A(\nabla f)\Big\rangle=&\Big\langle \alpha(\alpha-1)f^{\alpha-2}\nabla f, f^{\frac{p}{2}-1}\nabla f + (p-2)f^{\frac{p}{2}-2}\langle\nabla f,\nabla u\rangle\nabla u\Big\rangle,
\end{align}
and
\begin{align}\label{equ2.7}
\begin{split}
\alpha f^{\alpha-1}\mathcal{L}(f) = & \alpha f^{\alpha-1}\Big(\left(\frac{p}{2}-1\right)f^{\frac{p}{2}-2}|\nabla f|^2 + f^{\frac{p}{2}-1}\Delta f + (p-2)\left(\frac{p}{2}-2\right)f^{\frac{p}{2}-3}\langle\nabla f,\nabla u\rangle^2\\
				&+ (p-2)f^{\frac{p}{2}-2}\langle\nabla\langle\nabla f,\nabla u\rangle,\nabla u\rangle + (p-2)f^{\frac{p}{2}-2}\langle\nabla f,\nabla u\rangle\Delta u\Big).
\end{split}
\end{align}
Combining (\ref{equ2.6}) and (\ref{equ2.7}) yields
		\begin{align}\label{equ2.9}
			\begin{split}
				\Lie(f^\alpha)= &\alpha\left(\alpha+\frac{p}{2}-2\right)f^{\alpha+\frac{p}{2}-3}|\nabla f|^2
				+\alpha f^{\alpha+\frac{p}{2}-2}\Delta f\\
				& + \alpha(p-2)\left(\alpha+\frac{p}{2}-3\right)f^{\alpha+\frac{p}{2}-4}\langle\nabla f,\nabla u\rangle^2\\
				& + \alpha(p-2)f^{\alpha+\frac{p}{2}-3}\langle\nabla\langle\nabla f,\nabla u\rangle,\nabla u\rangle
				+\alpha(p-2)f^{\alpha+\frac{p}{2}-3}\langle\nabla f,\nabla u\rangle\Delta u.
			\end{split}
		\end{align}
		
On the other hand, by the definition of the $p$-Laplacian we have
\begin{align*}
\langle\nabla\Delta_p u,\nabla u\rangle &= \left(\frac{p}{2}-1\right)\left(\frac{p}{2}-2\right)f^{\frac{p}{2}-3}\langle\nabla f,\nabla u\rangle^2 + \left(\frac{p}{2}-1\right)f^{\frac{p}{2}-2}\langle\nabla\langle\nabla f,\nabla u\rangle,\nabla u\rangle\\
			&\quad+ \left(\frac{p}{2}-1\right)f^{\frac{p}{2}-2}\langle\nabla f,\nabla u\rangle\Delta u + f^{\frac{p}{2}-1}\langle\nabla\Delta u,\nabla u\rangle.
\end{align*}
Thus, the last term of (\ref{equ2.9}) can be rewritten as
\begin{align}\label{last}
\begin{split}
\alpha(p-2)f^{\alpha+\frac{p}{2}-3}\langle\nabla f,\nabla u\rangle\Delta u
&= 2\alpha f^{\alpha-1}\langle\nabla\Delta_p u,\nabla u\rangle - 2\alpha f^{\alpha+\frac{p}{2}-2}\langle\nabla\Delta u,\nabla u\rangle\\
&\quad- \alpha(p-2)\left(\frac{p}{2}-2\right)f^{\alpha+\frac{p}{2}-4}\langle\nabla f,\nabla u\rangle^2\\
&\quad- \alpha(p-2)f^{\alpha+\frac{p}{2}-3}\langle\nabla\langle\nabla f,\nabla u\rangle,\nabla u\rangle.
\end{split}
\end{align}
By (\ref{last}) and the following Bochner formula $$\frac{1}{2}\Delta f = |\nabla\nabla u|^2 + \ric(\nabla u,\nabla u) + \langle\nabla\Delta u,\nabla u\rangle,$$
we have
		\begin{align*}
			\mathcal{L}(f^{\alpha}) = &\alpha\left(\alpha+\frac{p}{2}-2\right)f^{\alpha+\frac{p}{2}-3}|\nabla f|^2
			+ 2\alpha f^{\alpha+\frac{p}{2}-2} \left(|\nabla\nabla u|^2 + \ric(\nabla u,\nabla u) \right)\\
			&+\alpha(p-2)(\alpha-1)f^{\alpha+\frac{p}{2}-4}\langle\nabla f,\nabla u\rangle^2 +2\alpha f^{\alpha-1}\langle\nabla\Delta_p u,\nabla u\rangle.
		\end{align*}
	\end{proof}
	
\begin{rem}
In the above lemma, we have used the following two facts:
\begin{enumerate}
\item By elliptic regularity theory, one know that $u$ is smooth away from $\{x:f(x)=0\}$. So the above equality holds true in the smooth sense.
			
\item The parameter $\alpha$ is employed to get a larger range of $(q, r)$ where the Cheng-Yau type gradient estimate holds for \eqref{equa1}. Obviously, the larger $\alpha$ is, the better the estimate of $\mathcal{L} (f^{\alpha})/\alpha$ is. Taking $\alpha\to\infty$ leads to the critical range of $(q,r)$ where the Liouville theorem is establishes by our method. When $\alpha=1$ and $\alpha=p/2$, \lemref{lem2.3} has been established in \cite{MR2518892} and \cite{MR4211885}.
\end{enumerate}
\end{rem}

\section{Proof of \thmref{t4}}
In this section, we generalized \cite[Theorem 15.1]{MR1879326} in $\R^n$ to non-compact Riemannian manifold with non-negative Ricci curvature.
\begin{thm}[{\thmref{t4}}]
Let $(M,g)$ be complete non-compact manifold with non-negative Ricci curvature. Suppose that $1<p<n$, $r>0, \, q\geq0,\, q+r>p-1$, and $a>0$. If
		\begin{align*}
			r+\frac{n-1}{n-p}q<\frac{n(p-1)}{n-p},
		\end{align*}
then the problem
		\begin{align}\label{veron}
			\begin{cases}
				\Delta_pv + av^r|\nabla v|^q\leq 0, &\quad\text{in}\,\, M;\\
				v> 0, & \quad\text{in}\,\, M
			\end{cases}
		\end{align}
admits no solution $v\in C^1(M)$.
\end{thm}
	
\begin{proof}
We show this theorem by almost the same arguments as in the proof of Theorem 15.1 in \cite{MR1879326} due to Mitidieri and Pohozaev. Let $\varphi\in C_0^1(M)$ be a non-negative cut-off function and $\alpha<0$ be small. Multiplying \eqref{veron} by $\varphi v^\alpha$ and integrating the inequality obtained over $M$, we arrive at
		\begin{align}\label{equa5}
			\int_Mav^{r+\alpha}|\nabla v|^q\varphi+|\alpha|\int_M|\nabla v|^pv^{\alpha-1}\varphi
			\leq \int_M |\nabla v|^{p-1}v^\alpha|\nabla \varphi|.
		\end{align}
It follows from \eqref{equa5} that
\begin{align*}
\int_Mav^{r+\alpha}|\nabla v|^q\varphi+|\alpha|\int_M|\nabla v|^pv^{\alpha-1}\varphi
\leq &\int_M |\nabla v|^{\frac{p}{s}}v^{\frac{\alpha-1}{s}}	\varphi^{\frac{1}{s}} v^{\frac{\alpha s+1-\alpha}{s}}
|\nabla \varphi||\nabla v|^{\frac{s(p-1)-p}{s}}\varphi^{-\frac{1}{s}},
\end{align*}
where $s>1$ is a constant which will be determined later. Now, applying Young inequality with exponents $s,\, t, \, \frac{1}{s}+\frac{1}{t}=1$ and $\epsilon>0$ for the integral function on the right hand side of the above inequality we have
\begin{align}\label{boho2}
\begin{split}
&\int_Mav^{r+\alpha}|\nabla v|^q\varphi+|\alpha|\int_M|\nabla v|^pv^{\alpha-1}\varphi\\
\leq &\frac{\epsilon}{s}\int_M|\nabla v|^pv^{\alpha-1}\varphi +\frac{\epsilon^{-\frac{1}{s-1}}}{t}\int_Mv^{\frac{\alpha s+1-\alpha}{s-1}}
|\nabla \varphi|^{\frac{s}{s-1}}|\nabla v|^{\frac{s(p-1)-p}{s-1}}\varphi^{-\frac{1}{s-1}},
\end{split}
\end{align}
where $\epsilon$ is small enough such that $|\alpha|\geq\frac{\epsilon}{s}$.
		
Hence, by H\"older inequality we can infer from \eqref{boho2} the following
\begin{align}\label{boho3}
\begin{split}
a\int_M v^{r+\alpha}|\nabla v|^q\varphi\leq &\frac{\epsilon^{-\frac{1}{s-1}}}{t}\int_Mv^{\frac{\alpha s+1-\alpha}{s-1}}
|\nabla \varphi|^{\frac{s}{s-1}}|\nabla v|^{\frac{s(p-1)-p}{s-1}}\varphi^{-\frac{1}{s-1}}\\
\leq &\frac{\epsilon^{-\frac{1}{s-1}}}{t}\left(\int_Mv^{\frac{\alpha s+1-\alpha}{s-1}\mu}|\nabla v|^{\frac{s(p-1)-p}{s-1}\mu} \varphi\right)^\frac{1}{\mu} \left(\int_M |\nabla \varphi|^{\frac{s\nu}{s-1}}\varphi^{1-\frac{s\nu}{s-1}}\right)^\frac{1}{\nu}.
\end{split}
\end{align}
Now we set $s$ and $\mu$ by
$$
\begin{cases}
\frac{\alpha s+1-\alpha}{s-1}\mu=r+\alpha,\\
\frac{s(p-1)-p}{s-1}\mu=q.
\end{cases}
$$
It is easy to check that
\begin{align}
s=\frac{pr+\alpha(p-q)+q}{r(p-1)+\alpha(p-1-q)}.
\end{align}
For $\alpha>1-p$, we can verify that
$$
\begin{cases}
pr+\alpha(p-q)+q>0,\\
r(p-1)+\alpha(p-1-q)>0
\end{cases}\quad\text{and}\quad s>\frac{p}{p-1}>1.
$$
It follows from \eqref{boho3} that
\begin{align}\label{boho4}
\int_M v^{r+\alpha}|\nabla v|^q\varphi\leq\left( \frac{\epsilon^{-\frac{1}{s-1}}}{at} \right)^\nu
\int_M|\nabla \varphi|^{\frac{s\nu}{s-1}}\varphi^{1-\frac{s\nu}{s-1}}.
\end{align}
		
Now we can choose $\gamma\in C^{\infty}_0(B_{2R}(o))$ such that
$$
\begin{cases}
\gamma(x)\equiv1, & \quad\text{in}\quad B_R(0),\\
1\geq \gamma(x)\geq 0,\quad |\gamma(x)|\leq\frac{C}{R}, &\quad\text{in}\quad B_{2R}(0)
\end{cases}
$$
and let $\varphi = \gamma^l$ with $l$ large enough, then we have
$$|\nabla \varphi|^{\frac{s\nu}{s-1}}\varphi^{1-\frac{s\nu}{s-1}}=l^{\frac{s\nu}{s-1}}\gamma^{l-\frac{s\nu}{s-1}}|\nabla \gamma|^{\frac{s\nu}{s-1}}\leq CR^{-\frac{s\nu}{s-1}}.$$
Since $\mathrm{Ric}_g\geq 0$, by the volume comparison theorem we conclude from \eqref{boho4} and the last inequality that
		\begin{align*}
			\int_Mv^{r+\alpha}|\nabla v|^q\varphi\leq CR^{n-\frac{s\nu}{s-1}}=CR^{n-\frac{s}{s-1}\frac{\mu}{\mu-1}}.
		\end{align*}
Direct computation shows that we can always choose $\alpha<0$ with $|\alpha|$ small enough such that
$$n-\frac{s}{s-1}\frac{\mu}{\mu-1}<0,$$
if
$$r(n-p)+q(n-1)<n(p-1).$$
Taking $R\to\infty$ leads to that
$$\int_Mv^{r+\alpha}|\nabla v|^q\varphi=0$$
which implies $v\equiv0$. Thus, we complete the proof of the theorem.
\end{proof}
	
Let $r>0, \, q\geq0$, and $p>1$. For small enough $\alpha<0$ we define
$$S_\alpha=\{u:M\to\mathbb{R}: u\geq0, \, u^{\alpha+r}|\nabla u|^q,\, u^{\alpha-1}|\nabla u|^p,\, u^{\alpha}|\nabla u|^{p-1}\in L^1_{loc}(M)\},$$
where $\nabla u$ is understood in the sense of distribution. Actually, we can also prove the following more general conclusions:
	
\begin{thm}
Let $(M,g)$ be complete non-compact manifold with non-negative Ricci curvature. Suppose that $1<p<n$, $r>0, \, q\geq0,\, q+r>p-1$, and $a>0$. If
\begin{align*}
r+\frac{n-1}{n-p}q<\frac{n(p-1)}{n-p},
\end{align*}
then the problem
\begin{align*}
\begin{cases}
\Delta_pv+av^r|\nabla v|^q\leq 0, &\quad\text{in}\,\, M;\\
v\geq 0, & \quad\text{in}\,\, M
\end{cases}
\end{align*}
admits no solution in $S_\alpha$ unless $v\equiv 0$.
\end{thm}
	
The proof of this theorem goes almost the same as the proof of Theorem 15.1 in \cite{MR1879326} except for we need to use the volume comparison theorem. So, we omit it.

\section{Gradient estimates and applications}
We divide the proof of \thmref{thm1} into three parts. In the first part, we derive a fundamental integral inequality on $f=|\nabla u|^2$, which will be used in the second and third parts. In the second part, we give an $L^{\beta}$-estimate of $f$ on a geodesic ball with radius $3R/4$, where $L^{\beta}$ norm of $f$ determines the initial state of the Nash-Moser iteration. Finally, we give a complete proof of our theorem by an intensive use of Nash-Moser iteration method.
	
\subsection{Integral inequality}\label{sect:3.1}
In order to establish the main theorem, we need to show some lemmas on integral estimates. We first introduce the following point-wise estimate for $\Lie(f^\alpha)$.
	\begin{lem}\label{lem41}
		Let $(M,g)$ be an $n$-dim complete Riemannian manifold with $\ric_g\geq -(n-1)\kappa g$ where $\kappa$ is a non-negative constant. Assume that $v$ is a positive solution to the equation (\ref{equa1}), $u = -(p-1)\log v$ and $f=|\nabla u|^2$. If $n,\, p,\, q,\, r$ and $a$ satisfy
		\begin{align*}
			a\left(\frac{n+1}{n-1}-\frac{q+r}{p-1}\right)\geq 0.
		\end{align*}
		or
		\begin{align*}
			1<\frac{q+r}{p-1}<\frac{n+3}{n-1},\quad  \forall a\in \R.
		\end{align*}
		Then there exists constant $\alpha_0=\alpha_0(n,p,q,r)>0$ such that the following inequality
		\begin{align*}
			\begin{split}
				\mathcal{L}(f^{\alpha_0 })\geq& 2\alpha_0 \beta_{n,p,q,r} f^{\alpha_0 +\frac{p}{2}} -2\alpha_0 (n-1)\kappa f^{\alpha_0 +\frac{p}{2}-1}- \alpha_0  a_1|\nabla f|f^{\alpha_0 +\frac{p-3}{2}}
			\end{split}
		\end{align*}
		holds point-wisely in $\{x:f(x)>0\}$, where $a_1 = \left|p -\frac{2(p-1)}{n-1}\right|$.
	\end{lem}
	
	\begin{proof}
		Let $\{e_1,e_2,\ldots, e_n\}$ be  an orthonormal frame of $TM$ on a domain with $f\neq 0$ such that $e_1=\frac{\nabla u}{|\nabla u|}$. We have $u_1 = f^{1/2}$. Direct computation shows
		\begin{align}\label{equ2.8}
			u_{11} = \frac{1}{2}f^{-1/2}f_1 = \frac{1}{2}f^{-1}\la\nabla u,\nabla f\ra,
		\end{align}
		and
		\begin{align}
			\label{equa2.9}
			|\nabla f|^2/f=4\sum_{i=1}^n u_{1i}^2 \geq 4u^2_{11}.
		\end{align}
		It follows from (\ref{equa2.3}) and (\ref{equ2.8}) that
		\begin{align}\label{equ2.10}
			\begin{split}
				\langle\nabla\Delta_p u,\nabla u\rangle=
				&\ \frac{p}{2}f^{\frac{p}{2}-1}\la\nabla f,\nabla u\ra+be^{cu}\frac{q}{2}f^{\frac{q}{2}-1}\la\nabla f,\nabla u\ra+bce^{cu}f^{\frac{q}{2}+1}\\
				=&\ pf^{\frac{p}{2} }u_{11}+bqe^{cu} f^{\frac{q}{2} }u_{11}+bce^{cu}f^{\frac{q}{2}+1}.
			\end{split}
		\end{align}
		By Cauchy inequality, we arrive at
		\begin{align}\label{equ2.11}
			|\nabla\nabla u|^2 \geq& u_{11}^2 + \sum_{i=2}u_{ii}^2 \geq u_{11}^2 +\frac{1}{n-1}\left(\sum_{i=2}u_{ii}\right)^2.
		\end{align}
		From now only, we always assume that $\alpha\geq 3/2$, which can guarantee $\alpha+p/2-1>0$. Substituting (\ref{equ2.8}), (\ref{equa2.9}), (\ref{equ2.10}) and (\ref{equ2.11}) into (\ref{bochner1}) and collecting terms about $u_{11}^2$, we obtained
		\begin{align}\label{equa3.5}
			\begin{split}
				\frac{f^{2-\alpha-\frac{p}{2}}}{2\alpha}
				\mathcal{L} (f^{\alpha}) \geq &\ (2\alpha-1)(p-1) u_{11}^2 + \ric(\nabla u,\nabla u)+\frac{1}{n-1}\left(\sum_{i=2}u_{ii}\right)^2\\
				&+ pf u_{11}+bqe^{cu} f^{1+\frac{p-q}{2} }u_{11}+bce^{cu}f^{\frac{q-p}{2}+2}.
			\end{split}
		\end{align}
		
		If we express the $p$-Lapalce in terms of $f$, we have
		\begin{align}\label{equ:2.12}
			\Delta_p u
			=&f^{\frac{p}{2}-1}\left((p-1)u_{11}+\sum_{i=2}^nu_{ii}\right).
		\end{align}
		Substituting (\ref{equ:2.12}) into equation (\ref{equa2.3}), we obtain
		\begin{align}\label{eq:2.13}
			(p-1)u_{11}+\sum_{i=2}^nu_{ii}^2
			=f +be^{cu}f^{\frac{q-p}{2}+1}.
		\end{align}
		It follows from (\ref{eq:2.13}) that
		\begin{align}\label{equ:2.13}
			\begin{split}
				\frac{1}{n-1}(\sum_{i=2}^nu_{ii}^2)=\frac{1}{n-1}&\left(f+be^{cu}f^{\frac{q-p}{2}+1}-(p-1)u_{11}\right)^2\\
				=\frac{1}{n-1}&\Big(f^2+b^2e^{2cu}f^{ q-p+2}+(p-1)^2u^2_{11}+2be^{cu}f^{\frac{q-p}{2}+2}\\
				&\ -2(p-1)fu_{11}-2b(p-1)e^{cu}f^{\frac{q-p}{2}+1}u_{11}\Big).
			\end{split}
		\end{align}
		
		It follows from (\ref{equa3.5}) and (\ref{equ:2.13}) that
		\begin{align}\label{equa2.15}
			\begin{split}
				\frac{f^{2-\alpha-\frac{p}{2}}}{2\alpha}
				\mathcal{L} (f^{\alpha})
				\geq
				&
				\left((2\alpha-1)(p-1)+\frac{(p-1)^2}{n-1}\right) u_{11}^2
				+
				\frac{f^2}{n-1}
				+
				\frac{b^2e^{2cu}f^{q-p+2}}{n-1}
				\\
				&
				+\left(p -\frac{2(p-1)}{n-1}\right)f u_{11}
				+ \ric(\nabla u,\nabla u)
				\\
				&
				+b\left(q-\frac{2(p-1)}{n-1}\right)e^{cu} f^{1+\frac{q-p}{2} }u_{11}
				+b\left(c+\frac{2}{n-1}\right)e^{cu}f^{\frac{q-p}{2}+2}.
			\end{split}
		\end{align}
		Making use of $a^2+2ab\geq -b^2$ , we have
		\begin{align}
			\label{equa2.16}
			\begin{split}
				&\left((2\alpha-1)(p-1)+\frac{(p-1)^2}{n-1}\right) u_{11}^2
				+
				b\left(q-\frac{2(p-1)}{n-1}\right)e^{cu} f^{1+\frac{q-p}{2} }u_{11}
				\\
				\geq
				&-
				\frac{\left(q(n-1)- 2(p-1) \right)^2}{ 4(n-1)(p-1)((2\alpha-1)(n-1)+ p-1)  }
				b^2e^{2cu}f^{q-p+2}.
			\end{split}
		\end{align}
		Substituting (\ref{equa2.16}) into (\ref{equa2.15}) yields
		\begin{align}
			\label{equa:3.11}
			\begin{split}
				\frac{f^{2-\alpha-\frac{p}{2}}}{2\alpha}
				\mathcal{L} (f^{\alpha}) \geq&\
				\frac{f^2}{n-1}
				+
				B_{n,p,q,\alpha} b^2e^{2cu}f^{q-p+2}
				+ \ric(\nabla u,\nabla u)
				\\
				&
				+\left(p -\frac{2(p-1)}{n-1}\right)f u_{11}
				+b\left(c+\frac{2}{n-1}\right)e^{cu}f^{\frac{q-p}{2}+2},
			\end{split}
		\end{align}
		where
		\begin{align}
			\label{equa:3.12}
			B_{n,p,q,\alpha} = \frac{1}{n-1}-
			\frac{\left(q(n-1)- 2(p-1) \right)^2}{ (n-1)(p-1)((2\alpha-1)(n-1)+ p-1)  }.
		\end{align}
		For any $n,p,q$ with $n>1, p>1$, we can always choose $\alpha$ large enough such that $B_{n,p,q,\alpha}$ is positive.
		
		\textbf{Case 1.}
		$$
		a\left(\frac{n+1}{n-1}-\frac{q+r}{p-1}\right)\geq 0.
		$$
		In this case, by the definition of $b, c$ in \eqref{defofbc} we have
		$$
		b\left(c+\frac{2}{n-1}\right)=(p-1)^{p-q-1}a\left(\frac{n+1}{n-1}-\frac{q+r}{p-1}\right)\geq 0.
		$$
		Hence,
		\begin{align*}
			\frac{f^{2-\alpha-\frac{p}{2}}}{2\alpha}
			\mathcal{L} (f^{\alpha}) \geq&\
			\frac{f^2}{n-1}
			+
			B_{n,p,q,\alpha} b^2e^{2cu}f^{q-p+2}
			+ \ric(\nabla u,\nabla u)
			\\
			&
			+\left(p -\frac{2(p-1)}{n-1}\right)f u_{11}.
		\end{align*}
		By \eqref{equa:3.12}, we have
		$$
		\lim_{\alpha\to\infty}B_{n,p,q,\alpha}=\frac{1}{n-1}>0.
		$$
		So there exists a constant $\alpha_1=\alpha_1(n,p,q)\geq 3/2$, for any  $\alpha\geq \alpha_1$, we have $B_{n,p,q,\alpha}>0$ and
		\begin{align}\label{eq412}
			\frac{f^{2-\alpha-\frac{p}{2}}}{2\alpha}
			\mathcal{L} (f^{\alpha}) \geq&
			\frac{f^2}{n-1}
			+ \ric(\nabla u,\nabla u)
			+\left(p -\frac{2(p-1)}{n-1}\right)f u_{11}.
		\end{align}
		
		\textbf{Case 2.}
		$$
		1<\frac{q+r}{p-1}<\frac{n+3}{n-1}.
		$$
		This condition is equivalent to
		\begin{align}
			\label{equ2.20}
			\begin{split}
				\frac{1}{n-1}-\frac{(n-1)}{4}\left(\frac{q+r}{p-1}-\frac{n+1}{n-1}\right)^2>0.
			\end{split}
		\end{align}
		Since
		$$
		B_{n,p,q} b^2e^{2cu}f^{q-p+2}+
		b\left(c+\frac{2}{n-1}\right)e^{cu}f^{\frac{q-p}{2}+2}\geq -\frac{1}{4B_{n,p,q,\alpha}}\left(c+\frac{2}{n-1}\right)^2f^2,
		$$
		we can conclude from (\ref{equa:3.11}) that
		\begin{align*}
			\frac{f^{2-\alpha-\frac{p}{2}}}{2\alpha}
			\mathcal{L} (f^{\alpha}) \geq&
			\left(\frac{1}{n-1}-\frac{1}{4B_{n,p,q,\alpha}}\left(c+\frac{2}{n-1}\right)^2\right)f^2
			\\
			&
			+ \ric\left(\nabla u,\nabla u\right)
			+\left(p -\frac{2(p-1)}{n-1}\right)f u_{11}.
		\end{align*}
		It follows from (\ref{equa:3.12}) and (\ref{equ2.20}) that
		$$
		\lim_{\alpha\to\infty}\frac{1}{n-1}-\frac{1}{4B_{n,p,q,\alpha}}\left(c+\frac{2}{n-1}\right)^2
		=\frac{1}{n-1}-\frac{(n-1)}{4}\left(\frac{q+r}{p-1}-\frac{n+1}{n-1}\right)^2>0.
		$$
		Then there exists a constant $\alpha_2=\alpha_2(n,p,q,r)\geq 3/2$, if we choose $\alpha\geq \alpha_2$, we have
		$$
		\beta_{n,p,q,r,\alpha}:=\frac{1}{n-1}-\frac{1}{4B_{n,p,q,\alpha}}\left(c+\frac{2}{n-1}\right)^2>0,
		$$
		then we have

		\begin{align}
			\label{equa:3.15a}
			\begin{split}
				\frac{f^{2-\alpha-\frac{p}{2}}}{2\alpha}
				\mathcal{L} (f^{\alpha}) \geq&\
				\beta_{n,p,q,r,\alpha} f^2
				+ \ric(\nabla u,\nabla u)
				+\left(p -\frac{2(p-1)}{n-1}\right)f u_{11}
				\\
				\geq&\
				\beta_{n,p,q,r,\alpha} f^2
				-(n-1)\kappa f
				-\frac{a_1}{2}f^{\frac{1}{2}} |\nabla f|,
			\end{split}
		\end{align}
		where
		$$
		a_1=\left|p -\frac{2(p-1)}{n-1}\right|.
		$$
		
		From now on, we choose $\alpha_0 = \max\{\alpha_1,\alpha_2\}$ and denote $\beta_{n,p,q,r}=\beta_{n,p,q,r,\alpha_0}$. We summarize the above two cases to obtain the following point-wise estimate
		\begin{align}\label{equa:3.15}
			\begin{split}
				\mathcal{L} (f^{\alpha_0 })
				\geq&
				2\alpha_0 \beta_{n,p,q,r} f^{\alpha_0 +\frac{p}{2} }
				-2\alpha_0 (n-1)\kappa  f^{\alpha_0 +\frac{p}{2}-1 }
				- \alpha_0  a_1   |\nabla f|f^{\alpha_0 +\frac{p-3}{2} },
			\end{split}
		\end{align}
		since
		$$
		\beta_{n,p,q,r}\leq \frac{1}{n-1}.
		$$
		Thus, we complete the proof of this lemma.
	\end{proof}
	
	\begin{lem}\label{lem3.2}
		Let $(M,g)$ be an $n$-dim complete Riemannian manifold with $\ric_g\geq -(n-1)\kappa g$ where $\kappa$ is a non-negative constant and $\Omega = B_R(o)\subset M$ be a geodesic ball. Assume that $v$ is a positive solution to the equation (\ref{equa1}), $u = -(p-1)\ln v$ and $f=|\nabla u|^2$. Then there exist constants $\beta_{n,p,q,r}$, $a_3,\, a_4$ and $a_5$ depending only on $n,p,q$ and $r$ such that for any $t\geq t_0 $, where $t_0$ is defined in \eqref{equa2.8}, and we have
		\begin{align*}
			\begin{split}
				& \beta_{n,p,q,r}\int_\Omega f^{\alpha_0+\frac{p}{2}+t}\eta^2 +
				\frac{a_3}{ t }e^{-t_0}V^{\frac{2}{n}}R^{-2}\left\|f^{\frac{\alpha_0+t-1}{2}+\frac{p}{4} }\eta\right\|_{L^{\frac{2n}{n-2}}}^2\\
				\leq\  & a_5t_0^2R^{-2} \int_\Omega f^{\alpha_0+\frac{p}{2}+t-1}\eta^2+\frac{a_4}{t }\int_\Omega f^{\alpha_0+\frac{p}{2}+t-1}|\nabla\eta|^2,
			\end{split}
		\end{align*}
		where $\eta\in C^{\infty}_0(\Omega,\R)$.
	\end{lem}
	
	\begin{proof}
		By the regularity theory on elliptic equations, $u$ is smooth away from $\{f=0\}$. So the both sides of \eqref{equa:3.15} are in fact smooth. Let $\epsilon>0$ and $\psi = f_\epsilon^{\alpha}\eta^2 $, where $f_\epsilon = (f-\epsilon)^+$, $\eta\in C^{\infty}_0(B_R(o))$ is non-negative and $\alpha>1$ which will be determined later. Now, we multiply the both sides of \eqref{equa:3.15} by $\psi$, integrate then the obtained inequality over $\Omega$ and take a direct computation to get
		\begin{align}
			\label{equa:3.17}
			\begin{split}
				&-\int_\Omega \alpha_0 tf^{\alpha_0+\frac{p}{2}-2}f_{\epsilon}^{t-1}|\nabla f|^2\eta^2
				+
				t\alpha_0(p-2)f^{\alpha_0+\frac{p}{2}-3}f_{\epsilon}^{t-1}\la\nabla f,\nabla u\ra^2\eta^2
				\\
				&-\int_\Omega2\eta\alpha_0 f^{\alpha+\frac{p}{2} -2}f_{\epsilon}^t\la\nabla f,\nabla\eta\ra+2\alpha_0\eta(p-2)f^{\alpha+\frac{p}{2}-3}f_{\epsilon}^t\la\nabla f,\nabla u\ra\la\nabla u,  \nabla\eta\ra
				\\
				\geq &
				2\beta_{n,p,q,r}\alpha_0 \int_\Omega f^{\alpha_0+\frac{p}{2}}f_{\epsilon}^t\eta^2  -2(n-1)\alpha_0\kappa\int_\Omega f^{\alpha_0+\frac{p}{2}-1}f_{\epsilon}^t\eta^2
				- a_1\alpha_0\int_\Omega f^{\alpha_0+\frac{p-3}{2}}f_{\epsilon}^t|\nabla f|\eta^2.
			\end{split}
		\end{align}
		It is easy to see that
		\begin{align}\label{2.24}
			f_{\epsilon}^{t-1}|\nabla f|^2 +(p-2)f_{\epsilon}^{t-1}f^{-1}\la\nabla f,\nabla u\ra^2\geq a_2 f_{\epsilon}^{t-1}|\nabla f|^2,
		\end{align}
		where $a_2 = \min\{1, p-1\}$, and
		\begin{align}\label{2.25}
			f_{\epsilon}^t\la\nabla f,\nabla\eta\ra+ (p-2)f_{\epsilon}^tf^{-1}\la\nabla f,\nabla u\ra\la\nabla u,  \nabla\eta\ra\geq -(p+1)f_{\epsilon}^t |\nabla f||\nabla\eta|.
		\end{align}
		Substituting \eqref{2.24}) and \eqref{2.25} into \eqref{equa:3.17} and letting $\epsilon\to 0$, we obtain
		\begin{align}
			\label{2.26}
			\begin{split}
				&2\beta_{n,p,q,r}\int_\Omega f^{\alpha_0+\frac{p}{2}+t}\eta^2
				+
				\int_\Omega a_2  tf^{\alpha_0+\frac{p}{2}+t-3}|\nabla f|^2\eta^2\\
				\leq &~
				2(n-1) \kappa\int_\Omega f^{\alpha_0+\frac{p}{2}+t-1}\eta^2
				+ a_1 \int_\Omega f^{\alpha_0+\frac{p-3}{2}+t }|\nabla f|\eta^2
				\\
				&
				+2 (p-1)\int_\Omega  f^{\alpha_0+\frac{p}{2}+t-2}|\nabla f||\nabla\eta|\eta,
		\end{split}\end{align}
		where we have divided the both sides of the inequality by $\alpha_0$.
		
		Using Cauchy-inequality, we have
		\begin{align}\label{2.27}
			\begin{split}
				a_1 f^{\alpha_0+\frac{p-3}{2}+t }|\nabla f|\eta^2\leq &\
				\frac{a_2t}{4}    f^{\alpha_0+\frac{p}{2}+t-3}|\nabla f|^2\eta^2
				+\frac{ a_1^2}{a_2t} f^{\alpha_0+\frac{p}{2}+t}\eta^2,\\
			\end{split}
		\end{align}
		and
		\begin{align}\label{2.27*}
			\begin{split}
				2(p+1)f^{\alpha_0+\frac{p}{2}+t-2}|\nabla f||\nabla\eta|\eta\leq &\
				\frac{a_2t}{4}    f^{\alpha_0+\frac{p}{2}+t-3}|\nabla f|^2\eta^2
				+\frac{4(p+1)^2 }{a_2t} f^{\alpha_0+\frac{p}{2}+t-1}|\nabla \eta|^2.
			\end{split}
		\end{align}
		Now we choose $t$ large enough such that
		\begin{align}\label{2.28}
			\frac{a_1^2}{a_2t}\leq  \beta_{n,p,q,r}.
		\end{align}
		It follows from (\ref{2.26}), (\ref{2.27}), \eqref{2.27*} and (\ref{2.28}) that
		\begin{align}
			\label{2.29}
			\begin{split}
				& \beta_{n,p,q,r}\int_\Omega f^{\alpha_0+\frac{p}{2}+t}\eta^2
				+
				\frac{a_2  t}{2}\int_\Omega f^{\alpha_0+\frac{p}{2}+t-3}|\nabla f|^2\eta^2 \\
				\leq  &\
				2(n-1) \kappa\int_\Omega f^{\alpha_0+\frac{p}{2}+t-1}\eta^2
				+\frac{4(p+1)^2 }{a_2t}  \int_\Omega f^{\alpha_0+\frac{p}{2}+t-1}|\nabla \eta|^2.
			\end{split}
		\end{align}
		
		On the other hand, we have
		\begin{align}\label{2.30}
			\begin{split}
				\left|\nabla \left(f^{\frac{\alpha_0+t-1}{2}+\frac{p}{4} }\eta \right)\right|^2\leq &\ 2\left|\nabla f^{\frac{\alpha_0+t-1}{2}+\frac{p}{4} }\right|^2\eta^2 +2f^{\alpha_0+t-1+\frac{p}{2}}|\nabla\eta|^2\\
				=&\ \frac{(\alpha_0+t+\frac{p}{2}-1)^2}{2}f^{\alpha_0+t+\frac{p}{2}-3}|\nabla f |^2\eta^2 +2f^{\alpha_0+t-1+\frac{p}{2}}|\nabla\eta|^2  .
			\end{split}
		\end{align}
		Now substituting \eqref{2.30} into \eqref{2.29} leads to
		\begin{align}\label{2.31}
			\begin{split}
				& \beta_{n,p,q,r} \int_\Omega f^{\alpha_0+\frac{p}{2}+t}\eta^2 + \frac{4a_2t}{(2\alpha_0+2t+p-2)^2}\int_\Omega \left|\nabla \left(f^{\frac{\alpha_0+t-1}{2}+\frac{p}{4} }\eta\right)\right|^2 \\
				\leq  &\ 2(n-1)\kappa  \int_\Omega f^{\alpha_0+t+\frac{p}{2}-1}\eta^2
				+ \frac{4(p-1)^2 }{a_2t} \int_\Omega f^{\alpha_0+\frac{p}{2}+t-1}|\nabla\eta|^2\\
				&\ +\frac{8a_2t}{(2\alpha_0+2t+p-2)^2}\int_\Omega f^{\alpha_0+t+\frac{p}{2}-1}|\nabla\eta|^2 .
			\end{split}
		\end{align}
		
		From now on we use $a_1,\, a_2,\, a_3\, \cdots$ to denote constants depending only on $n,\, p,\, q$ and $r$. We choose $a_3$ and $a_4$ such that
		\begin{align}\label{2.32}
			\frac{a_3}{t}\leq  \frac{4a_2t}{(2\alpha_0+2t+p-2)^2}\quad\text{and}\quad\frac{8a_2t}{(2\alpha_0+2t+p-2)^2}+\frac{4(p-1)^2}{a_2t} \leq\frac{a_4}{t},
		\end{align}
		since $t$ satisfies (\ref{2.28}).
		
		It follows from (\ref{2.31}) and (\ref{2.32}) that
		\begin{align}
			\label{2.33}
			\begin{split}
				& \beta_{n,p,q,r}\int_\Omega f^{\alpha_0+\frac{p}{2}+t}\eta^2
				+
				\frac{a_3}{t}\int_\Omega   \left|\nabla \left(f^{\frac{\alpha_0+t-1}{2}+\frac{p}{4} }\eta\right)\right|^2 \\
				\leq  &\
				2(n-1)\kappa  \int_\Omega f^{\alpha_0+t+\frac{p}{2}-1}\eta^2
				+
				\frac{a_4 }{t} \int_\Omega f^{\alpha_0+\frac{p}{2}+t-1}\left|\nabla\eta\right|^2 .
			\end{split}
		\end{align}
		Noting that Saloff-Coste's Sobolev inequality implies
		$$e^{-c_0(1+\sqrt{\kappa}R)}V^{\frac{2}{n}}R^{-2}\left\|f^{\frac{\alpha_0+t-1}{2}+\frac{p}{4} }\eta\right\|_{L^{\frac{2n}{n-2}}(\Omega)}^2\leq \int_{\Omega}\left|\nabla \left(f^{\frac{\alpha_0+t-1}{2}+\frac{p}{4} }\eta\right)\right|^2+R^{-2}\int_\Omega f^{ \alpha_0+t +\frac{p}{2} -1 }\eta ^2,$$
		furthermore we can infer from \eqref{2.33} that
		\begin{align}\label{3.32}
			\begin{split}
				& \beta_{n,p,q,r}\int_\Omega f^{\alpha_0+\frac{p}{2}+t}\eta^2
				+
				\frac{a_3}{t }e^{-c_0(1+\sqrt{\kappa}R)}V^{\frac{2}{n}}R^{-2}\left\|f^{\frac{\alpha_0+t-1}{2}+\frac{p}{4} }\eta\right\|_{L^{\frac{2n}{n-2}}}^2\\
				\leq  &\
				2(n-1)\kappa  \int_\Omega f^{\alpha_0+t+\frac{p}{2}-1}\eta^2+\frac{a_4}{ t }\int_\Omega f^{\alpha_0+t+\frac{p}{2}-1}|\nabla\eta|^2
				+\frac{a_3}{ t }\int_\Omega R^{-2}f^{ \alpha_0 +\frac{p}{2}+t-1 }\eta ^2.
			\end{split}
		\end{align}

		Now we choose
		\begin{align}\label{equa2.8}
			t_0 = c_{n,p,q,r}\left(1+\sqrt{\kappa} R\right) \quad\text{and}\quad c_{n,p,q,r}=\max\left\{c_0+1, \frac{a_1^2}{a_2\beta_{n,p,q,r}} \right\},
		\end{align}
		and pick $t$ such that $t\geq t_0$. Since
		\begin{align*}
			2(n-1)\kappa  R^2\leq\frac{ 2(n-1)}{c_{n,p,q,r}^2}t_0^2\quad \text{ and}\quad \frac{a_3}{t}\leq \frac{a_3}{c_{n,p,q,r}},
		\end{align*}
		there exists $a_5 = a_5(n,p,q,r)>0$ such that
		\begin{align}\label{a5}
			2(n-1)\kappa  R^2+\frac{a_3}{t}\leq a_5t_0^2 = a_5c^2_{n,p,q,r}\left(1+\sqrt{\kappa} R\right)^2.
		\end{align}
		Immediately, it follows from \eqref{3.32} and \eqref{a5} that
		\begin{align}\label{2.34}
			\begin{split}
				& \beta_{n,p,q,r}\int_\Omega f^{\alpha_0+\frac{p}{2}+t}\eta^2 + \frac{a_3}{ t }e^{-t_0}V^{\frac{2}{n}}R^{-2}\left\|f^{\frac{\alpha_0+t-1}{2}+\frac{p}{4} }\eta\right\|_{L^{\frac{2n}{n-2}}}^2\\
				\leq\  & a_5t_0^2R^{-2} \int_\Omega f^{\alpha_0+\frac{p}{2}+t-1}\eta^2+\frac{a_4}{t }\int_\Omega f^{\alpha_0+\frac{p}{2}+t-1}|\nabla\eta|^2.
			\end{split}
		\end{align}
		Thus, we finish the proof.
	\end{proof}

	\subsection{\texorpdfstring{$L^{\beta}$}\, bound of gradient in a geodesic ball with radius $3R/4$}\label{sect3.2}
	
	\begin{lem}\label{lpbound}
		Let $(M,g)$ be an $n$-dimensional complete Riemannian manifold with Ricci curvature $\ric_g\geq -(n-1)\kappa g$ where $\kappa$ is a non-negative constant and $\Omega = B_R(o)\subset M$ is a geodesic ball. Assume that $f$ is the same as in the above lemma. For a given positive number $$\beta = \left(\alpha_0+t_0+\frac{p}{2}-1\right)\frac{n}{n-2},$$
		there exists a constant $a_8 = a_8(n,p, q)>0$ such that
		\begin{align}\label{lpbpund}
			\|f \|_{L^{\beta}(B_{3R/4}(o))}\leq a_8V^{\frac{1}{\beta}} \frac{t_0^2}{ R^2},
		\end{align}
		where $V$ is the volume of geodesic ball $B_R(o)$.
	\end{lem}
	
	\begin{proof}
		We set $t=t_0$ in \eqref{2.34} and obtain
		\begin{align}\label{equation:3.31}
			\begin{split}
				& \beta_{n,p,q,r}\int_\Omega f^{\alpha_0+\frac{p}{2}+t_0}\eta^2 + \frac{a_3}{ t_0 } e^{-t_0}V^{\frac{2}{n}}R^{-2}\left\|f^{\frac{\alpha_0+t_0-1}{2}+\frac{p}{4}}\eta\right\|_{L^{\frac{2n}{n-2}}}^2\\
				\leq  &\ a_5t_0^2R^{-2} \int_\Omega f^{\alpha_0+\frac{p}{2}+t_0-1}\eta^2+\frac{a_4}{t_0 }\int_\Omega f^{\alpha_0+\frac{p}{2}+t_0-1}|\nabla\eta|^2.
			\end{split}
		\end{align}
		From (\ref{equation:3.31}) we know that, if
		$$
		f\geq  \frac{2a_{5}t_0^2}{\beta_{n,p,q,r}R^2},
		$$
		then there holds true
		$$
		a_5t_0^2R^{-2} \int_\Omega f^{\alpha+\frac{p}{2}+t_0-1}\eta^2\leq\frac{\beta_{n,p,q,r}}{2}\int_\Omega f^{\alpha_0+\frac{p}{2}+t_0}\eta^2.
		$$
		Now we denote $$\Omega_1 = \left\{f<  \frac{2a_{5}t_0^2}{\beta_{n,p,q,r}R^2} \right\},$$
		then $\Omega$ can be decomposed into two parts $\Omega_1$ and $\Omega_2$, i.e., $\Omega=\Omega_1\cup\Omega_2$ where $\Omega_2$ is the complement of $\Omega_1$. Hence, we have
		\begin{align}\label{2.36}
			\begin{split}
				a_5t_0^2R^{-2} \int_\Omega f^{\alpha_0+\frac{p}{2}+t_0-1}\eta^2
				=&\ a_5t_0^2R^{-2} \int_{\Omega_1} f^{\alpha_0+\frac{p}{2}+t_0-1}\eta^2
				+a_5t_0^2R^{-2} \int_{\Omega_2} f^{\alpha_0+\frac{p}{2}+t_0-1}\eta^2
				\\
				\leq&\
				\frac{a_5t_0^2}{R^2} \left(\frac{2a_{5}t_0^2}{\beta_{n,p,q,r}R^2}\right)^{\alpha_0+\frac{p}{2} +t_0-1 }V
				+
				\frac{\beta_{n,p,q,r} }{2}\int_\Omega f^{\alpha_0+\frac{p}{2}+t_0}\eta^2,
			\end{split}
		\end{align}
		where $V$ is the volume of $B_R(o)$. It follows from (\ref{2.34}) and (\ref{2.36})
		\begin{align}
			\label{2.37}
			\begin{split}
				&\frac{\beta_{n,p,q,r}}{2}\int_\Omega f^{\alpha_0+\frac{p}{2}+t_0}\eta^2
				+
				\frac{a_3}{ t_0 }e^{-t_0}V^{\frac{2}{n}}R^{-2}\left\|f^{\frac{\alpha_0+t_0-1}{2}+\frac{p}{4} }\eta\right\|_{L^{\frac{2n}{n-2}}}^2\\
				\leq \ &
				\frac{a_5t_0^2}{R^2} \left(\frac{2a_{5}t_0^2}{\beta_{n,p,q,r}R^2}\right)^{\alpha_0+\frac{p}{2} +t_0-1 }V
				+\frac{a_4}{ t_0 }\int_\Omega f^{\alpha_0+\frac{p}{2}+t_0-1}|\nabla\eta|^2.
			\end{split}
		\end{align}
		
		Now we choose $\gamma\in C^{\infty}_0(B_R(o))$ satisfying
		$$
		\begin{cases}
			0\leq\gamma(x)\leq 1,\quad |\nabla\gamma(x)|\leq\frac{C }{R}, &\forall x\in B_R(o);\\
			\gamma(x)\equiv 1, & \forall x\in B_{ {3R}/{4}}(o),
		\end{cases}
		$$
		and let $\eta = \gamma^{ \alpha_0 + \frac{p}{2}+t_0}$. Then, we have
		\begin{align}\label{314}
			a_4R^2 |\nabla\eta|^2
			\leq
			a_4 C^2 \left( t_0+ \frac{p}{2}+\alpha_0\right )^2\eta ^{\frac{2\alpha_0+2t_0+p-2}{\alpha_0+p/2+t_0}}
			\leq
			a_{6}t^2_0\eta^{\frac{2\alpha_0+2t_0+p-2}{\alpha_0+p/2+t_0}}.
		\end{align}
		By H\"older inequality and Young inequality, we have
		\begin{align}
			\label{2.39}
			\begin{split}
				\frac{a_4}{t_0}\int_{\Omega}f^{\frac{p}{2}+\alpha_0+t_0-1}|\nabla\eta|^2
				\leq\  &\frac{a_6t_0}{R^2} \int_{\Omega}f^{\frac{p}{2}+\alpha_0+t_0-1}
				\eta^{\frac{2\alpha_0+p+2t_0-2}{\alpha_0+p/2+t_0}}
				\\
				\leq\
				&
				\frac{a_6t_0}{R^2}  \left(\int_{\Omega}f^{\alpha_0+t_0+ \frac{p}{2} }\eta^2\right)^{\frac{\alpha_0+p/2+t_0-1}{\alpha+p/2+t_0}}V^{\frac{ 1}{\alpha_0+t_0+ p/2 }}\\
				\leq \ &
				\frac{\beta_{n,p,q,r}}{2}\left[\int_{\Omega}f^{ \alpha_0+t_0+\frac{p}{2} }\eta^2
				+
				\left(\frac{2a_{6}t_0 }{\beta_{n,p,q,r}R^2}\right)^{ \alpha_0+t_0+\frac{p}{2} }V\right].
			\end{split}
		\end{align}
		We conclude from (\ref{2.37}) and (\ref{2.39}) that
		\begin{align}
			\label{2.40}
			\begin{split}
				&\left(\int_{\Omega}f^{\frac{n(p/2+\alpha_0+t_0-1)}{n-2}}\eta^{\frac{2n}{n-2}}\right)^{\frac{n-2}{n}}
				\\
				\leq\
				&
				\frac{t_0}{a_3} e^{t_0}V^{1-\frac{2}{n}}R^2
				\left[\frac{2a_5t_0^2}{R^2} \left(\frac{2a_{5}t_0^2}{\beta_{n,p,q,r}R^2}\right)^{\alpha_0+t_0+\frac{p}{2}-1 }
				+
				\frac{a_{6}t^2_0}{R^2} \left(\frac{2a_{6}t_0 }{\beta_{n,p,q,r}R^2}\right)^{\alpha_0+t_0+\frac{p}{2} -1 }\right]
				\\
				\leq\
				&
				a_7^{t_0}e^{t_0}V^{1-\frac{2}{n}}t_0^3
				\left( \frac{t_0^2}{ R^2}\right)^{\alpha_0+t_0+\frac{p}{2} -1},
			\end{split}
		\end{align}
		where the constant $a_7$ depends only on $n,\,p,\, q$ and $r$, and satisfies
		$$
		a_7^{t_0} \geq \frac{2a_5}{a_3}\left(\frac{4a_5}{\beta_{n,p,q,r}}\right)^{\alpha_0+t_0+\frac{p}{2}-1}
		+\frac{a_6}{a_3}\left(\frac{4a_6}{\beta_{n,p,q,r}t_0}\right)^{\alpha_0+t_0+\frac{p}{2}-1}.
		$$
		Here we have used the fact $t_0\geq1$.
		
		Taking $\frac{1}{\alpha_0+t_0+p/2-1}$ power of the both sides of (\ref{2.40}) gives
		\begin{align}
			\left\|f\eta^{\frac{2}{\alpha_0+t_0+p/2-1}}\right\|_{L^{\beta}(\Omega)}\leq a_7^{\frac{t_0}{\alpha_0+t_0+p/2-1}}V^{1/\beta}
			t_0^{\frac{3}{\alpha_0+t_0+p/2-1}}\frac{t_0^2}{ R^2}\leq a_8V^{\frac{1}{\beta}} \frac{t_0^2}{ R^2},
		\end{align}
		where $a_8$ depends only on $n,\ p,\ q,\ r$ and satisfies
		$$
		a_8 \geq a_7^{\frac{t_0}{\alpha_0+t_0+p/2-1}}t_0^{\frac{3}{\alpha_0+t_0+p/2-1}}.
		$$
		Since $\eta\equiv1$ in $B_{3R/4}$, we obtain that
		$$
		\|f \|_{L^{\beta}(B_{3R/4}(o))}\leq a_8V^{\frac{1}{\beta}} \frac{t_0^2}{ R^2}.
		$$
		Thus, the proof of this lemma is finished.
	\end{proof}

	\subsection{Moser iteration}\label{sec3.3}
	Now we turn to providing the proof of \thmref{thm1}.
	\begin{proof}
		We discard the first term in (\ref{2.34}) to obtain
		\begin{align}\label{2.42}
			\frac{a_3}{ t }e^{-t_0}V^{\frac{2}{n}}R^{-2}\left\|f^{\frac{\alpha_0+t-1}{2}+\frac{p}{4} }\eta\right\|_{L^{\frac{2n}{n-2}}}^2\leq
			\frac{a_5\alpha_0^2}{R^2}  \int_\Omega f^{\alpha_0+\frac{p}{2}+t-1}\eta^2+\frac{a_4}{ t }\int_\Omega f^{\alpha_0+\frac{p}{2}+t-1}|\nabla\eta|^2.
		\end{align}
		For $k=1,\, 2,\, \cdots$, let $$r_k = \frac{R}{2}+\frac{R}{4^k}$$ and $\Omega_k = B_{r_k}(o)$. Choosing a sequence of cut-off function $\eta_k\in C^{\infty}(\Omega_k)$ which satisfy
		\begin{align}
			\begin{cases}
				0\leq \eta_k(x)\leq1, \quad |\nabla\eta_k(x)|\leq \frac{C4^k}{R}, & \forall x\in B_{r_{k+1}}(o);\\
				\eta_k(x)\equiv 1, &\forall x\in B_{r_{k+1}}(o),
			\end{cases}
		\end{align}
		where $C$ is a constant that does not depend on $(M,g)$ and equation \eqref{equa1}, and substituting $\eta_k$ into (\ref{2.42}) instead of $\eta$, we arrive at
		\begin{align*}
			a_3e^{-t_0}V^{\frac{2}{n}} \left\|f^{\frac{\alpha_0+t-1}{2}+\frac{p}{4} }\eta_k\right\|_{L^{\frac{2n}{n-2}}(\Omega_k)}^2
			\leq \ &
			a_5t_0^2t\int_{\Omega_k} f^{\alpha_0+\frac{p}{2}+t-1}\eta_k^2+ a_4R^2\int_{\Omega_k} f^{\alpha_0+\frac{p}{2}+t-1}\left|\nabla\eta_k\right|^2\\
			\leq\ & \left(a_5t_0^2t + C^216^k\right)\int_{\Omega_k} f^{\alpha_0+\frac{p}{2}+t-1}.
		\end{align*}
		Now we choose $\beta_1=\beta$, $\beta_{k+1}=n\beta_k/(n-2)$, $k=1,\,2,\, \cdots$, and let $t=t_k$ such that
		$$t_k+\frac{p}{2}+\alpha_0-1=\beta_k,$$
		then we have
		\begin{align*}
			a_3 \left(\int_{\Omega_k}f^{\beta_{k+1}}\eta_k^\frac{2n}{n-2}\right)^{\frac{n-2}{n}} \leq\
			e^{t_0}V^{-\frac{2}{n}}\left(a_5t_0^2\left(t_0+\alpha_0+\frac{p}{2}-1\right)\left(\frac{n}{n-2}\right)^k + C^216^k\right)\int_{\Omega_k} f^{\beta_k}.
		\end{align*}
		Since $n/(n-2)<16, ~\forall n>2$, there exists some constant $a_9 $ such that
		\begin{align}
			\label{eq:2.45}
			\left(\int_{\Omega_k}f^{\beta_{k+1}}\eta_k^\frac{2n}{n-2}\right)^{\frac{n-2}{n}}
			\leq \ a_9t_0^316^k e^{ t_0}V^{-\frac{2}{n}}  \int_{\Omega_k} f^{\beta_k}.
		\end{align}
		Taking $1/\beta_k$ power of the both sides of (\ref{eq:2.45}) respectively and using the fact $\eta_k\equiv1\in\Omega_{k+1}$, we obtain
		\begin{align}\label{2.46}
			\begin{split}
				\|f\|_{L^{\beta_{k+1}}(\Omega_{k+1})}
				\leq & \left(a_9t_0^316^k e^{ t_0}V^{-\frac{2}{n}}\right)^{\frac{1}{\beta_k}} 16 ^{\frac{k}{\beta_k}}\|f\|_{L^{\beta_k}(\Omega_k)}.
			\end{split}
		\end{align}
		Noting
		$$
		\sum_{k=1}^{\infty}\frac{1}{\beta_k} =\frac{n}{2\beta_1} \quad\text{and}\quad \sum_{k=1}^{\infty}\frac{k}{\beta_k}=\frac{n^2}{4\beta_1},
		$$
		we have
		\begin{align}
			\label{2.47}
			\|f\|_{L^{\infty}(B_{R/2}(o))}
			\leq \
			a_{10}  V^{-\frac{1}{\beta}} \|f\|_{L^{\beta_1}(B_{3R/4}(o))},
		\end{align}
		where
		$$
		a_{10} \geq \left(a_9t_0^316^k e^{ t_0}\right)^{\frac{n}{2\beta_1 }} 16 ^{\frac{n^2}{4\beta_1}}.
		$$
		By (\ref{lpbpund}), we obtain
		\begin{align}
			\label{3.44}
			\|f\|_{L^{\infty}(B_{R/2}(o))}
			\leq\ a_{11}\frac{(1+\sqrt{\kappa}R)^2}{R^2},
		\end{align}
where $a_{11} = a_{10}a_8c_{n,p,q,r}$. Reminding $f=|\nabla u|^2$ and $u=-(p-1)\log v$, we obtain the desired estimate. Thus, we complete the proof of \thmref{thm1}.
\end{proof}

\subsection{Proof of \thmref{thm1.4}}
In this case, we have $\kappa=0$. By \corref{cor1.3}, we have
\begin{align}\label{3.45}
\frac{|\nabla v(x_0)|}{v(x_0)}\leq \sup_{B(x_0,\frac{R}{2})}\frac{|\nabla v|}{v}\leq \frac{C(n,p,q,r)}{R}, \quad \forall x_0\in M.
\end{align}
Letting $R\to\infty$ in \eqref{3.45}, we obtain
$$
\nabla v(x_0)=0,\quad\forall x_0\in M.
$$
thus $v$ is a constant.
	\qed

\subsection{Proof of \thmref{thm1.5} }
For any $x\in B_{R/2}(o)\subset M$,  by \corref{cor1.3}, we have
\begin{align}\label{3.48}
|\nabla \ln v(x)|\leq c(n,p,q)\frac{ 1+\sqrt{\kappa}R }{R }.
\end{align}
Choosing a minimizing geodesic $\gamma(t)$ with arc length parameter connecting $o$ and $x$:
	$$
	\gamma:[0,d]\to M,\quad\gamma(0)=o, \quad \gamma(d)=x.
	$$
where $d=d(x, o)$ is the geodesic distance from $o$ to $x$, we have
	\begin{align}
		\ln v(x)-\ln v(o)=\int_0^d\frac{d}{dt}\ln v\circ\gamma(t)dt.
	\end{align}
Since $|\gamma'|=1$, we have
\begin{align}
\left|\frac{d}{dt}\ln v\circ\gamma(t)\right|\leq |\nabla \ln v||\gamma'(t)| \leq c(n,p,q)\frac{ 1+\sqrt{\kappa}R }{R }.
\end{align}	
Thus it follows from $d\leq R/2$ and the above inequality that
$$
e^{-c(n,p,q)(1+\sqrt{\kappa}R)/2}\leq v(x)/v(o)\leq e^{c(n,p,q)(1+\sqrt{\kappa}R)/2}.
$$
So, for any $x, y\in B_{R/2}(o)$ we have
$$
v(x)/v(y)\leq e^{c(n,p,q)(1+\sqrt{\kappa}R)}.
$$

Now suppose $v$ is a global positive solution of equation \eqref{equa1} on $M$. Letting $R\to\infty$ in \eqref{3.48}, we obtain that
	$$
	|\nabla \ln v(x)|\leq c(n,p,q)\sqrt{\kappa}, \quad \forall x\in M.
	$$
For any $y\in M$, choose a minimizing geodesic $\gamma(t)$ parameterized by arc length which connects $x$ and $y$:
	$$
	\gamma:[0,d]\to M,\quad\gamma(0)=x, \quad \gamma(d)=y.
	$$
where $d=d(x,y)$ is the distance from $x $ to $y$. Then, we have
	\begin{align}\label{3.49}
		\ln v(y)-\ln v(x)=\int_0^d\frac{d}{dt}\ln v\circ\gamma(t)dt.
	\end{align}
Since $|\gamma'(t)|=1$, we have
	\begin{align}\label{3.50}
		\left|\frac{d}{dt}\ln v\circ\gamma(t)\right|\leq |\nabla \ln v||\gamma'(t)| = c(n,p,q)\sqrt{\kappa}.
	\end{align}
It follows from (\ref{3.49}) and (\ref{3.50}) that
	\begin{align}\label{3.51}
		v(y)/v(x)\leq e^{c(n,p,q)\sqrt{\kappa}d(x, y)}.
	\end{align}
Thus we finish the proof by (\ref{3.51}).

\begin{cor}\label{cor34}
Let $(M, g)$ be a complete Riemannian manifold and $\Omega\subset M$ be a domain. Assume that $v\in C^2(\Omega)$ is a positive solution of
	\begin{align}\label{eq3.53}
		\Delta_pv+av^r|\nabla v|^q=0\quad\text{in}\quad\Omega,
	\end{align}
If
	\begin{align}
		a\geq0 \quad\text{ and }\quad q+r <\frac{n+3}{n-1}(p-1),
	\end{align}
	or
	\begin{align}
		a\leq 0 \quad\text{ and }\quad q+r>p-1,
	\end{align}
then for any $x\in\Omega$, we have
	\begin{align}\label{eq3.54}
		|\nabla\ln v(x)|\leq C_{n,p,q}\left(d(x,\partial\Omega)^{-1}+\sqrt{\kappa}\right)
	\end{align}
\end{cor}

\begin{proof}
	Denote $R = d(x,\partial\Omega)$. By \thmref{thm1}, we have
	$$
	|\nabla \ln v(x)|\leq \sup_{y\in B_{\frac{R}{2}}(x)}|\nabla \ln v(y)|\leq c_{n,p,q}\left(d(x,\partial\Omega)^{-1}+\sqrt{\kappa}\right).
	$$
\end{proof}

\begin{cor}
Let $\Omega$ is a domain containing $o$ and let $R>0$ be such that $d(o,\partial\Omega) \geq2R$. Assume $q+r <\frac{n+3}{n-1}(p-1)$. Then, for any positive solution $u\in C^2(\Omega\setminus\{o\})$ of \eqref{equa1.5} in $\Omega\setminus\{o\}$ there holds true
\begin{align}\label{355}
v(x)\leq \sup_{|y|=R} v(y)\cdot e^{c_{n,p,q}\sqrt{\kappa}|R-d(x,o)|}\left(\frac{R}{d(x,o)}\right)^{c_{n,p,q}}, \quad\forall x\in B_{R}\setminus \{o\}.
\end{align}
\end{cor}
\begin{proof}
	Let $\gamma(t)$ be the geodesic ray starting from $o$ to $x$ and $\gamma(r)=x$.
	Let $y$ lies in the geodesic ray such that $d(o,y) = R$, i.e. $\gamma(R)=y$, then $B_R(y)\subset\Omega$ and $d(\gamma(tR+(1-t)r), \partial B_R(y))=tR+(1-t)r$ for any $0\leq t\leq 1$. Since $v$ is a solution of \eqref{eq3.53} in $B_{R}(y)$, by \eqref{eq3.54} we have
	\begin{align*}
		|\ln v(x)-\ln v(y)|=&\ \left|\int^1_0\frac{d}{dt}\ln v\circ\gamma(tR+(1-t)r)dt\right|
		\\
		\leq &\ |R-r|\int^1_0|\nabla \ln v\circ\gamma(tR+(1-t)r) |dt\\
		\leq &\  c_{n,p,q}|R-r|\int^1_0\left((tR+(1-t)r)^{-1}+\sqrt{\kappa}\right)dt
		\\
		=&\  c_{n,p,q}(\ln R-\ln r+\sqrt{\kappa}|R-r|),
	\end{align*}
	i.e.
	\begin{align*}
		\ln v(x)
		\leq &\ c_{n,p,q}(\ln R-\ln r+\sqrt{\kappa}|R-r|)+\sup_{|y|=R}\ln v(y).
	\end{align*}
	\eqref{355} follows from the above inequality.
\end{proof}

\subsection{The case $\dim(M)=2$.}
In the proof of \thmref{thm1} above, we have used Saloff-Coste's Sobolev inequality on the embedding $W^{1,2}(B)\hookrightarrow L^{\frac{2n}{n-2}} (B)$ on a manifold. Since the Sobolev exponent $2^*=2n/(n-2)$ requires $n>2$, we did not consider the case $n=2$ in \thmref{thm1}. In this subsection, we will explain briefly that \thmref{thm1} can also be established when $n=2$.
	
When $\dim M=n=2$, we need the special case of  Saloff-Coste's Sobolev theorem, i.e., \lemref{salof}. For any $n'> 2$, there holds
	$$
	\|f\|_{L^{\frac{2n'}{n'-2}}(B)}^2\leq e^{c_0(1+\sqrt{\kappa}R)}V^{-\frac{2}{n'}}R^2\left(\int_B|\nabla f|^2+R^{-2}f^2\right)
	$$
for any $f\in C^{\infty}_0(B)$.
	
For example, we can choose $n'=4$, then we have	
\begin{align*}
\left(\int_{\Omega}f^{2\alpha+p}\eta^{4}\right)^{2}\leq e^{c_0(1+\sqrt{\kappa} R)}V^{-\frac{1}{2}}\left(R^2\int_{\Omega}\left|\nabla \left(f^{\frac{p}{4}+\frac{\alpha}{2}}\eta\right)\right|^2+\int_{\Omega}f^{\frac{p}{2}+\alpha}\eta^2\right).
\end{align*}
By the above inequality and \eqref{equation:3.31}, we can deduce the following integral inequality by almost the same method as in \lemref{lem3.2}.
\begin{align} \label{equation3.19}
\begin{split}
& \beta_{n,p,q,r}\int_\Omega f^{\alpha_0+\frac{p}{2}+t}\eta^2 + \frac{a_3}{ t }e^{-t_0}V^{\frac{1}{2}}R^{-2}\left\|f^{\frac{\alpha_0+t-1}{2}+\frac{p}{4} }\eta\right\|_{L^{4}}^2\\
\leq\  & a_5t_0^2R^{-2} \int_\Omega f^{\alpha_0+\frac{p}{2}+t-1}\eta^2+\frac{a_4}{t }\int_\Omega f^{\alpha_0+\frac{p}{2}+t-1}|\nabla\eta|^2,
\end{split}
\end{align}
where $\alpha_0$ is the same as  $\alpha_0$ defined in Section \ref{sect:3.1},  but the constants $a_i$, $i=3, \, 4,\, 5$, may differ from those defined in Section \ref{sect:3.1}.
	
By repeating the same procedure as in Section \ref{sect3.2}, we can deduce from \eqref{equation3.19} the $L^{\beta}$-bound of $f $ in a geodesic ball with radius $3R/4$
	\begin{align}\label{equation:3.20}
		\|f \|_{L^{\beta}(B_{3R/4}(o))}\leq a_8V^{\frac{1}{\beta}} \frac{t_0^2}{ R^2},
	\end{align}
where $\beta=p+2\alpha_0+2t_0-2$.
	
For the Nash-Moser iteration, we set $$\beta_l = 2^l(\alpha_0+t_0+p/2-1)$$ and $\Omega_l$ by the similar way with that in Section \ref{sec3.3}, and can obtain the following inequality
	\begin{align}\label{equation:3.21}
		\|f\|_{L^{\infty}(B_{\frac{R}{2}}(o))}\leq\ &a_{11} V^{-\frac{1}{\beta}}\|f\|_{L^{\beta}(B_{\frac{3R}{4}}(o))}.
	\end{align}
Combining \eqref{equation:3.20} and \eqref{equation:3.21}, we finally obtain the Cheng-Yau type gradient estimate. Harnack inequality and Liouville type results follow from the Cheng-Yau type gradient estimate, here we omit the details.

\section{Some applications for the equations on $\Omega\subset\mathbb R^n$.}
\subsection{Some known results and notions on the equation \eqref{equa1} on $\R^n$}

Due to the special importance of the semilinear case, we give separate statements (and proofs). It is easy to see that in the case $p=2$ the equation \eqref{equa1} on $\R^n$ is invariant under the action of the transformations $T_\sigma$ defined for $\sigma>0$ by
	$$T_\sigma(v)(x)=\sigma^{\frac{2-q}{q+r-1}}v(\sigma x).$$
If we looks for radial positive solutions under the form $u(x)=\Lambda |x|^{-\gamma}$, we find that,
if $q<2$ and $q + r - 1>0$, then $\gamma=\gamma_{q,r}=\frac{2-q}{q+r-1}$ and
$$\Lambda=\Lambda_{n,q,r}=\gamma_{q,r}^{\frac{1-q}{q+r-1}}\left(n-\frac{q+2r}{q+r-1}\right)^{\frac{1}{q+r-1}}.$$
However, this last quantity exists if and only if the exponents belong to the supercritical range, that is, when
$$q(n-1) + r(n-2) > n.$$
In the subcritical range of exponents, that is, when
$$q(n-1) + r(n-2) < n,$$
we have mentioned in the previous section that Mitidieri and Pohozaev (Theorem 15.1 in \cite{MR1879326}) proved that if $r>0$ and $q\geq0$ satisfy $q+r>1$ and
\begin{align}\label{vercondi1}
r(n-2)+q(n-1)<n,
\end{align}
then all nonnegative functions $v\in C^1(\mathbb{R}^n)$ satisfying
$$-\Delta v\geq |\nabla v|^qv^r\geq 0$$
are constant.

In the subcritical range of exponents, Bidaut-V\'eron, Garcia-Huidobro and V\'eron \cite{MR3959864} proved recently that Serrin’s classical results (see Theorem 1 of \cite{Ser1}, \cite{Ser2}) can be applied and obtain a local Harnack inequality and an a priori estimate for positive solution $v$ in $B_R \setminus \{0\}$ (see Theorem A in \cite{MR3959864}). They also adopted the well-known Bernstein method and the Keller-Osserman comparison method to discuss this equation in the supercritical range of exponents. The a priori gradient estimates, which are one of the main results of their paper (see Theorem B, page 1489 in \cite{MR3959864}), were derived. Concretly, they considered the following equation
\begin{align}\label{equa1.5}
\begin{cases}
\Delta  v + |\nabla v|^qv^r = 0, &\text{in }\, \R^n;\\
v>0, & \text{in }\,\R^n,
\end{cases}
\end{align}
where $0\leq q <2$ and $r\geq 0$. Their results can be summarized by the following:

\begin{theorem*}
(A). Let $\Omega\subset\mathbb{R}^n$ be a domain containing $0$, $n \geq 3$, $r \geq 0$, $0 \leq q < 2$, and assume that
$(n-2)r + (n-1)q < n$ holds. If $v\in C^2(\Omega\setminus\{0\})$ is a positive solution of \eqref{equa1.5} in $\Omega\setminus\{0\}$, then estimate
\begin{equation}\label{Ver}
v(x)+|x||\nabla v(x)| < c|x|^{2-n}
\end{equation}
holds in a neighborhood of $0$ with a constant $c$ depending on $v$.

(B). Let $n \geq 2$, $0\leq q<2$, and $r\geq 0$ be such that $r+q-1>0$. If $v$ is a positive solution of \eqref{equa1.5} in $B_R(0)\subset \mathbb{R}^n$ and either of the assumptions
\begin{enumerate}
\item $1\leq r< \frac{n+3}{n-1}\quad \text{and}\quad q + r - 1 < \frac{4}{n-1}$,
			
\item $0 \leq r<1 \quad \text{and} \quad q + r - 1 < \frac{(r+1)^2}{r(n-1)}$,
\end{enumerate}
is fulfilled, then there exist positive constants $a^*= a^*(n, q, r)$ and $c_1 = c_1(n, q, r)$ such that
		$$|\nabla v^{a^*}(0)|\leq c_1R^{-1-a^*\frac{2-q}{q+r-1}}.$$
\end{theorem*}
	
Note that $\frac{4}{n-1}\leq \frac{(r+1)^2}{r(n-1)}$ always holds. The authors of \cite{MR3959864} also pointed out that the value of the exponent $a^*$ is not easy to compute.
	
As a consequence of the above theorem, they also derived some Liouville properties for \eqref{equa1} on $\R^n$ or \eqref{equa1.5}, which can be stated as follows: under the assumptions on $n$, $q$ and $r$ of the above theorem, any positive solution of \eqref{equa1.5} in $\R^n$ is constant.

{\it It is well-known that from Cheng-Yau logarithmic gradient estimate, one can easily obtain the Harnack inequality which enables us to derive the local estimates of solutions. A natural problem remaining in \cite{MR3959864} is for us how to obtain the concise and universal Cheng-Yau logarithmic gradient estimates for the positive solutions to \eqref{equa1} on $\R^n$ or \eqref{equa1.5} with $p>1$ since it is hard to compute $a^*$ appeared in the above estimates obtained in \cite{MR3959864} and the constant in \eqref{Ver} depends on $u$ even in the semilinear case $p=2$.

On the other hand, one wants to know whether or not the conditions supposed in the above theorem can be weaken? For instance, whether or not the condition (1) in the above theorem
		$$0\leq q <2, \quad r\geq 0, \quad q+r>1,\quad 1\leq r< \frac{n+3}{n-1}\quad \text{and}\quad q + r < \frac{n+3}{n-1}$$
		can be weaken to
		\begin{align}\label{vercondi2}
			q+r<\frac{n+3}{n-1}
		\end{align}
where $q$ and $r$ are real constants.}

\subsection{Some new results on $\Omega\subset\mathbb{R}^n$: Proofs of \thmref{t7} and Corollaries \ref{tui1}-\ref{tui3}}\

As an application of our main results in the previous section, actually we have answered partially the above problems. See \thmref{thm1} and Corollary \ref{cor1.3}. Moreover, for the equation \eqref{equa1.5} we obtain the following:

\begin{thm}\label{t51}(Theorem \ref{t7})
(i). Let $\Omega\subset\mathbb{R}^n$ be a domain containing $0$ and $p=2$. Assume that $a$, $q$ and $r$ fulfill \eqref{cond3}, i.e.,
$$a\geq 0 \quad\text{ and }\quad q+r <\frac{n+3}{n-1}.$$
If $v\in C^2(\Omega\setminus\{0\})$ is a positive solution of \eqref{equa1} in $\Omega\setminus\{0\}$, then, in the case $n\geq3$ there holds true in a neighborhood of $0$
\begin{equation*}\label{Ver}
v(x)+|x||\nabla v(x)| \leq c|x|^{2-n},
\end{equation*}
where $c$ is a constant depending on $u$; and in the case $n=2$ there holds in a neighborhood of $0$
\begin{equation*}\label{Ver}
v(x)+|x||\nabla v(x)| \leq c\log\frac{1}{|x|},
\end{equation*}
where $c$ is a constant depending on $v$.

(ii). Let $\Omega\subset\mathbb{R}^n$ be a domain containing $0$ and $p=2$. Assume that $a$, $q$ and $r$ fulfill \eqref{cond4}, i.e.,
$$a\leq 0 \quad\text{ and }\quad q+r>p-1.$$
If $v\in C^2(\Omega\setminus\{0\})$ is a positive solution of \eqref{equa1} in $\Omega\setminus\{0\}$, then, in the case $n\geq3$ there exists  a constant $c$ depending on $v$ such that there holds in a neighborhood of $0$
\begin{equation*}
v(x) \geq c|x|^{2-n};
\end{equation*}
and in the case $n=2$ there exists  a constant $c$ depending on $v$ such that there holds in a neighborhood of $0$
\begin{equation*}\label{Ver}
v(x) \geq c\log\frac{1}{|x|}.
\end{equation*}
\end{thm}

\begin{proof}
(i). Since $a\geq 0$ and $p=2$, we can see easily that the spherical average $\bar{v}$ of $v$ on $\{x: |x|=r\}$ is superharmonic, if $v$ is a positive solution to $$\Delta v + a|\nabla v|^q v^r=0.$$
Hence there exists some $m \geq 0$ such that in a neighborhood of $0$
$$\bar{v}(r)\leq m r^{2-n}\quad\mbox{for}\,\, n\geq 3\quad\mbox{and}\quad \bar{v}(r)\leq m \log\frac{1}{|x|}\quad\mbox{for}\,\, n=2.$$
On the other hand, $v$ verifies the Harnack inequality obtained in \thmref{thm1.5}
$$\sup_{|x|=r} v(x)\leq K\inf_{|x|=r}v(x)\quad\mbox{for}\,\, r\in (0,\, R],$$
where $K>0$ is some positive constant. Combined with the Harnack inequality, this yields that in a neighborhood of $0$
$$v(x)\leq C |x|^{2-n}\quad\mbox{for}\,\, n\geq 3\quad\mbox{and}\quad v(x)\leq C \log\frac{1}{|x|} \quad\mbox{for}\,\, n=2.$$

(ii). As $a\leq 0$ and $p=2$, we can see easily that the spherical average $\bar{v}$ of $v$ on $\{x: |x|=r\}$ is subharmonic, if $v$ is a positive solution to
$$\Delta v + a|\nabla v|^q v^r=0.$$
Hence there exists some $m \geq 0$ such that in a neighborhood of $0$
$$\bar{v}(r)\geq m r^{2-n}\quad\mbox{for}\,\, n\geq 3\quad\mbox{and}\quad \bar{v}(r)\geq m \log\frac{1}{|x|}\quad\mbox{for}\,\, n=2.$$

On the other hand, from \thmref{thm1.5} we know that $v$ verifies the following Harnack inequality
$$\sup_{|x|=r} v(x)\leq K\inf_{|x|=r}v(x)\quad\mbox{for}\,\, r\in (0,\, R],$$
where $K>0$ is some positive constant. Combined with the Harnack inequality, this yields that in a neighborhood of $0$
$$v(x)\geq C |x|^{2-n}\quad\mbox{for}\,\, n\geq 3\quad\mbox{and}\quad v(x)\geq C \log\frac{1}{|x|} \quad\mbox{for}\,\, n=2.$$
Thus, we finish the proof.
\end{proof}

\begin{cor}\label{cor35}(Corollary \ref{tui1})
Let $\Omega\subset \mathbb R^n$ be a domain. If $v\in C^2(\Omega)$ is a positive solution of
\begin{align}\label{eq3.58}
\begin{cases}
&\Delta_pv+av^r|\nabla v|^q=0\quad\text{in}\quad\Omega;\\
&a\geq0 \quad\text{ and }\quad q+r <\frac{n+3}{n-1}(p-1),\\
&\text{or} \quad a\leq 0 \quad\text{ and }\quad q+r>p-1 ,
\end{cases}
\end{align}
then for any $x\in\Omega$, we have
\begin{align}\label{eq3.57}
|\nabla\ln v(x)|\leq C_{n,p,q}d(x, \partial\Omega)^{-1}.
\end{align}
\end{cor}

\begin{proof}
Denote $R = d(x,\partial\Omega)$. By \thmref{thm1}, we have
	$$
	|\nabla \ln v(x)|\leq \sup_{y\in B_{\frac{R}{2}}(0)}|\nabla \ln v(y)|\leq c_{n,p,q}\frac{1}{R}=c_{n,p,q}d(x,\partial\Omega)^{-1}.
	$$
\end{proof}

By \corref{cor35}, we can estimate the order of singularity of the solution in punctured domain.
\begin{cor}(Corollary \ref{tui2})
Let $\Omega$ be a domain containing $0$ and let $R>0$ be such that $d(0,\partial\Omega) \geq2R$. Then any $v\in C^2(\Omega\setminus\{0\})$ solution of equation \eqref{eq3.58} with the conditions in \corref{cor35} satisfies
\begin{align}\label{eq3.55}
v(x)\leq \sup_{|y|=R} v(y)\cdot \left(\frac{R}{|x|}\right)^{c_{n,p,q}},\quad \forall x\in B_{R}\setminus \{0\}.
\end{align}
\end{cor}

\begin{proof}
Let $y = \frac{R}{|x|}x$, then $B_R(y)\subset\Omega$ and for any $0\leq t\leq 1$,
$$d(tx+(1-t)y, \partial B_R(y))=t|x|+(1-t)R.$$
Since $v$ is a solution of \eqref{eq3.58} in $B_{R}(y)$, by \eqref{eq3.57} we have
	\begin{align*}
		|\ln v(x)-\ln v(y)|=&\ \left|\int^1_0\frac{d}{dt}\ln v(tx+(1-t)y)dt\right|
		\\
		\leq &\ |x-y|\int^1_0|\nabla \ln v(tx+(1-t)y) |dt\\
		\leq &\  c_{n,p,q}|x-y|\int^1_0(t|x|+(1-t)R)^{-1}dt
		\\
		=&\  c_{n,p,q}(\ln R-\ln|x|),
	\end{align*}
	i.e.
	\begin{align*}
		\ln v(x)
		\leq &\  c_{n,p,q}(\ln R-\ln|x|)+\sup_{|y|=R}\ln v(y).
	\end{align*}
Immediately, \eqref{eq3.55} follows from the above inequality.
\end{proof}

\begin{cor}(Corollary \ref{tui3})
	Assume $\Omega$ is a bounded domain with a $C^2$ boundary. Then there exists $\delta^*>0$ such that if we denote $\Omega_{\delta^*}:=\{z\in\Omega:d(z, \partial\Omega)\leq\delta^*\}$, any solution $v\in C^2(\Omega)$ \eqref{eq3.58} in $\Omega$ satisfies
	\begin{align*}
		v(x)\leq \sup_{d(z^*,\partial\Omega)=\delta^*}  v(z^*)\left(\frac{\delta^*}{d(x,\partial\Omega)} \right)^{c_{n,p,q}}.
	\end{align*}
\end{cor}
\begin{proof}
We denote by $\delta^*$ the maximal $r>0$ such that any boundary point $a\in\partial\Omega$ belongs to the boundary of a ball $B_{r}(a_i)\subset\overline{\Omega}$ with radius $r$ and to simultaneously the boundary of another ball $B_r(a_s)\subset\overline{\Omega}^c$ with radius $r$. In other words,
$$\partial\Omega\cap\partial B_{r}(a_i)\cap \partial B_{r}(a_s)=\{a\}.$$
If $x\in\Omega_{\delta^*}$, we denote by $\sigma(x)$ its projection onto $\partial\Omega$ and $\mathbf{n}_{\sigma(x)}$ the outward normal unit vector to $\partial\Omega$ at $\sigma(x)$ and $z^* = \sigma(x)-\delta_*\mathbf{n}_{\sigma(x)}$, then we have
\begin{align*}
|\ln v(x)-\ln v(z^*)|\leq  &\left|\int^1_0\frac{d}{dt}\ln v(tx+(1-t)z^*)dt\right| \\
= &\left|\int^1_0\langle\nabla\ln v(tx+(1-t)z^*), x-z^*\rangle dt\right|\\
\leq &c_{n,p,q}|d(x)-\delta^*|\int_0^1(td(x,\partial\Omega)+(1-t)\delta^*)^{-1}dt\\
=& c_{n,p,q}(\ln \delta^*-\ln d(x,\partial\Omega))
\end{align*}
If follows that
\begin{align*}
	\ln v(x)
	\leq &\  c_{n,p,q}(\ln \delta^*-\ln d(x,\partial\Omega)+\sup_{d(z^*,\partial\Omega)=R}\ln v(z^*)
\end{align*}
and
\begin{align*}
	v(x)\leq \sup_{d(z^*,\partial\Omega)=\delta^*}  v(z^*)\left(\frac{\delta^*}{d(x,\partial\Omega)} \right)^{c_{n,p,q}}.
\end{align*}
Thus, we finish the proof.
\end{proof}

\section{Global gradient estimate}
By the local gradient estimate of positive solution $v$ to the equation \eqref{equa1}, we have
\begin{align*}
	\sup_{B_{R/2}(o)}\frac{|\nabla v|}{v}\leq C(n,p,q,r)\frac{1+\sqrt{\kappa}R}{R}.
\end{align*}
If $v$ is an entire solution, letting $R\to\infty$ in the above local gradient estimate, obviously we obtain
\begin{align*}
	\frac{|\nabla v|}{v}(x)\leq C(n,p,q,r)\sqrt{\kappa}, \quad\forall x\in M.
\end{align*}
Sung and Wang \cite{MR3275651} have ever studied the sharp bound of $C(n,p,q,r)$ for the case $p$-harmonic function.  Motivated by \cite{MR3275651} we would like to give an explicit expression of the above $C(n,p,q,r)$ in this section.
\begin{thm}[\thmref{t10}]
Let $(M,g)$ be a complete non-compact Riemannian manifold with $\mathrm{Ric}_g\geq-(n-1)\kappa$. Assume that $v$ is an entire positive solution of \eqref{equa1} in $M$. Then,
\begin{enumerate}
\item if
		\begin{align}\label{61}
			a\left(\frac{n+1}{n-1}-\frac{q+r}{p-1}\right)\geq 0,
		\end{align}
there holds
		\begin{align*}
			\frac{|\nabla v|}{v}(x)\leq \frac{(n-1)\sqrt{\kappa}}{p-1}, \quad\forall x\in M;
		\end{align*}
\item if
\begin{align}\label{62}
1<\frac{q+r}{p-1}<\frac{n+3}{n-1},\quad  \forall a\in \R,
\end{align}
there holds
\begin{align*}
\frac{|\nabla v|}{v}(x)\leq \frac{2\sqrt{\kappa}}{(p-1)\sqrt{\left(\frac{q+r}{p-1}-1\right)\left(\frac{n+3}{n-1}-\frac{q+r}{p-1}\right)}}, \quad\forall x\in M.
\end{align*}
\end{enumerate}
\end{thm}

We will prove \thmref{t10} by the following two lemmas
\begin{lem}
Let $(M,g)$ be a complete non-compact Riemannian manifold with $\mathrm{Ric}_g\geq-(n-1)\kappa$. Assume that $v$ is an entire positive solution of \eqref{equa1} in $M$.
\begin{enumerate}
\item If
\begin{align*}
a\left(\frac{n+1}{n-1}-\frac{q+r}{p-1}\right)\geq 0,
\end{align*}
denote $y_1=(n-1)^2\kappa$ and define $\omega = (f-y_1-\delta)^+$ for any $\delta>0$, then we have
		\begin{align*}
			\mathcal{L}(\omega^\alpha)\geq 	2\alpha\omega^{\alpha-1}  \left(k_1\omega-k_2|\nabla \omega|\right),
		\end{align*}
where $\alpha>2$, $k_1$ and $k_2$ are two positive constants depending on $n,p,q,r$ and $\kappa$.
\item If
\begin{align*}
1<\frac{q+r}{p-1}<\frac{n+3}{n-1},\quad  \forall a\in \R,
\end{align*}		
denote
\begin{align*}
y_2 = \frac{4\kappa}{\left(\frac{n+3}{n-1}-\frac{q+r}{p-1}\right)\left( \frac{q+r}{p-1}-1\right)}
\end{align*}
and define $\omega = (f-y_2-\delta)^+$ for any $\delta>0$, then, for $\alpha$ large enough we have
\begin{align*}
\mathcal{L}(\omega^\alpha)\geq 	2\alpha\omega^{\alpha-1}  \left(l_1\omega-l_2|\nabla \omega|\right),
\end{align*}
where $l_1$ and $l_2$ are two positive constants depending on $n,p,q,r$ and $\kappa$.
\end{enumerate}
\end{lem}

\begin{proof}
According to \lemref{lem41}, we know that, if \eqref{61} is satisfied, then there holds
\begin{align*}
		\frac{f^{2-\alpha-\frac{p}{2}}}{2\alpha}
		\mathcal{L} (f^{\alpha}) \geq&
		\frac{f^2}{n-1}
		-(n-1)\kappa f
		-\frac{a_1}{2} f^{\frac{1}{2}}|\nabla f|.
\end{align*}
On $\{f\geq y_1+\delta\}$ we have
\begin{align*}
\mathcal{L}(\omega^\alpha) = \di(\alpha\omega^{\alpha-1}A(\nabla f))=\alpha(\alpha-1)\omega^{\alpha-2}\langle\nabla\omega, A(\nabla f)\rangle+\alpha\omega^{\alpha-1}\mathcal{L}(f).
\end{align*}
In the above expression, it is easy to see that $\nabla\omega=\nabla f$ in $\{f>y_1+\delta\}$, but $\nabla\omega$ causes a distribution on $\{f=y_1+\delta\}$. Now, if we assume $\alpha>2$, then the distribution caused by $\nabla\omega$ on $\{f=y_1+\delta\}$ is eliminated by $\omega$ since $\omega=0$  on $\{f=y_1+\delta\}$. Since $f\geq \omega$ and $\langle\nabla f, A(\nabla f)\rangle\geq 0$, we can infer easily from the above identity
	\begin{align*}
		\mathcal{L}(\omega^\alpha)
		\geq &\ \omega^{\alpha-1}\left(\alpha(\alpha-1)f^{-1}\langle\nabla f, A(\nabla f)\rangle+\alpha \mathcal{L}(f)\right)=\ \frac{\omega^{\alpha-1}}{f^{\alpha-1}}\mathcal{L}(f^\alpha).
	\end{align*}
It follows from \eqref{eq412} and the above inequality that
	\begin{align*}
		\mathcal{L} (f^{\alpha})\geq &\ 2\alpha\omega^{\alpha-1}f^{\frac{p}{2}-1}  \left(\frac{f^2}{n-1}-(n-1)\kappa f-\frac{a_1}{2}f^{\frac{1}{2}}|\nabla \omega|\right).
	\end{align*}
Noting $f\geq y_1+\delta = (n-1)^2\kappa+\delta$, we have
\begin{align*}
\mathcal{L} (f^{\alpha})\geq &\ 2\alpha\omega^{\alpha-1}f^{\frac{p-1}{2}}  \left(f^{\frac{1}{2}}\left(\frac{f}{n-1}-(n-1)\kappa \right)-\frac{a_1}{2}|\nabla \omega|\right)\\
\geq& \ 2\alpha\omega^{\alpha-1}f^{\frac{p-1}{2}}  \left(f^{\frac{1}{2}} \frac{\delta}{n-1} -\frac{a_1}{2}|\nabla \omega|\right).
\end{align*}
Using the fact $$y_1+\delta\leq f\leq c(n,p,q,r)\sqrt{\kappa}$$ on $\{f\geq y_1+\delta\}$, we have
\begin{align*}
\mathcal{L}(f^\alpha)\geq 	2\alpha\omega^{\alpha-1}  \left(k_1\omega-k_2|\nabla \omega|\right), \quad\forall \alpha >2,
\end{align*}
where $k_1$ and $k_2$ are two positive constants depending on $n,\, p,\, q,\, r,\, \kappa$ and $\delta$.
	
If \eqref{62} is satisfied, by \eqref{equa:3.15a} we have
\begin{align*}
\frac{f^{2-\alpha-\frac{p}{2}}}{2\alpha}\mathcal{L} (f^{\alpha})\geq &\ \beta_{n,p,q,r,\alpha}f^2-(n-1)\kappa f -\frac{a_1}{2}f^{\frac{1}{2}}|\nabla f|,
\end{align*}
where
$$
\beta_{n,p,q,r} =\lim_{\alpha\to\infty}\beta_{n,p,q,r,\alpha}= \frac{1}{n-1}-\frac{n-1}{4}\left(\frac{q+r}{p-1}-\frac{n+1}{n-1}\right)^2>0.
$$
Then, by taking a computation we have on $\{f\geq y_2+\delta\}$
\begin{align*}
\mathcal L(\omega^\alpha) \geq & \ 2\alpha\omega^{\alpha-1}f^{\frac{p-1}{2}}  \left(f^{\frac{1}{2}}\left(\beta_{n, p, q, r,\alpha} f -(n-1)\kappa \right)-\frac{a_1}{2}|\nabla \omega|\right).
\end{align*}
Since
	$$
	\lim_{\alpha\to\infty}\beta_{n, p, q, r,\alpha} f  -(n-1)\kappa=\left(\frac{1}{n-1}-\frac{n-1}{4}\left(\frac{q+r}{p-1}-\frac{n+1}{n-1}\right)^2\right)\delta=\beta_{n,p,q,r}\delta,
	$$
	we can choose $\alpha$ large enough such that
	$$
	\beta_{n, p, q, r,\alpha} f  -(n-1)\kappa>\frac{1}{2}\beta_{n,p,q,r}\delta.
	$$
It follows
\begin{align*}
\mathcal L(\omega^\alpha) \geq & \ 2\alpha\omega^{\alpha-1}f^{\frac{p-1}{2}}  \left(f^{\frac{1}{2}}\frac{\beta_{n,p,q,r}\delta}{2}-\frac{a_1}{2}|\nabla \omega|\right)\\
		=& 2\alpha\omega^{\alpha-1} \left(f^{\frac{p}{2}} \frac{\beta_{n,p,q,r}\delta}{2}-\frac{a_1}{2}f^{\frac{p-1}{2}}|\nabla \omega|\right).
\end{align*}
Using the fact $$y_2+\delta\leq f\leq c(n,p,q,r)\sqrt{\kappa}\quad \text{and}\quad f\geq\omega$$ on $\{f\geq y_2+\delta\}$, we have
\begin{align*}
\mathcal{L}(\omega^\alpha)\geq& 2\alpha\omega^{\alpha-1} \left(f^{\frac{p}{2}-1} \frac{\beta_{n,p,q,r}\delta}{2}\omega-\frac{a_1}{2}f^{\frac{p-1}{2}}|\nabla \omega|\right)\\
\geq &2\alpha\omega^{\alpha-1}(l_1\omega-l_2|\nabla\omega|)
\end{align*}
where $l_1$ and $l_2$ are two positive constants depending on $n,\,p,\,q,\,r,\,\delta$ and $\kappa$.
\end{proof}

\begin{lem}\label{lem42}
Let $(M,g)$ be an $n$-dim($n\geq 2$) complete Riemannian manifold satisfying $\mathrm{Ric}_g\geq-(n-1)\kappa g$ for some constant $\kappa\geq0$. Assume that $v\in C^1(M)$ is an entire positive solution of \eqref{equa1}. If $\omega = (f-y)^+$ for some $y>0$ satisfies the following inequality
$$	 		
\mathcal L(\omega^\alpha)\geq \omega^{\alpha-1}(b_1\omega-b_2|\nabla\omega|)
$$
where $b_1$, $b_2$ and $\alpha$ are some positive constants, then $\omega\equiv 0$, i.e.,  $f\leq y$.
\end{lem}

\begin{proof}
Let $\eta\in C^{\infty}_0(M,\mathbb R)$ be a cut-off function to be determine later. We choose $\eta^2\omega^\gamma$ as a test function, then we have
\begin{align*}
\int_M\mathcal L(\omega^\alpha)\omega^{\gamma}\eta^2\geq \int_M2\alpha\omega^{\alpha+\gamma-1}  \left(b_1\omega-b_2|\nabla \omega|\right)\eta^2.
\end{align*}
	
We omit the term $f^{\frac{p}{2}-1}$ since $f$ is uniform bounded on $M$. By integration by parts, we obtain
\begin{align*}
&\int_M2\alpha\omega^{\alpha+\gamma-1}  \left(b_1\omega-b_2|\nabla \omega|\right)\eta^2\\
\leq& -\int_M\alpha\gamma\omega^{\alpha+\gamma-2}\eta^2(|\nabla\omega|^2+(p-2)f^{-1}\langle\nabla\omega, \nabla u\rangle^2)\\
&-\int_M2\alpha\eta\omega^{\alpha+\gamma-1}\left(\langle\nabla\omega, \nabla \eta \rangle+(p-2)f^{-1}\langle\nabla\omega,\nabla\eta\rangle\langle \nabla u,\nabla\eta\rangle\right).
\end{align*}
It follows
\begin{align*}
&\int_M2 \omega^{\alpha+\gamma-1}  \left(b_1\omega-b_2|\nabla \omega|\right)\eta^2+a_1\int_M \gamma\omega^{\alpha+\gamma-2}\eta^2|\nabla\omega|^2\\
\leq &2(p+1)\int_M\omega^{\alpha+\gamma-1}|\nabla\omega||\nabla \eta|  \eta.
\end{align*}
Cauchy inequality implies
\begin{align*}
2b_2 \omega^{\alpha+\gamma-1}  |\nabla \omega| \eta^2\leq&\ b_2\eta^2\omega^{\alpha+\gamma-2}\left(\frac{|\nabla\omega|^2}{\epsilon}+\epsilon\omega^2\right) ;\\
2 \eta(p+1) \omega^{\alpha+\gamma-1}|\nabla\omega||\nabla \eta|
\leq& \ (p+1)  \omega^{\alpha+\gamma-2}\left(\frac{\eta^2|\nabla\omega|^2}{\epsilon}+\epsilon|\nabla\eta|^2\omega^2\right).
\end{align*}
If we choose $\epsilon$ such that
\begin{align*}
\frac{b_2+p+1}{\epsilon}=a_1\gamma,
\end{align*}
then we have
\begin{align*}
&\int_M2b_1 \omega^{\alpha+\gamma}\eta^2\leq\int_M  \epsilon(p+1)  \omega^{\alpha+\gamma }|\nabla\eta|^2 +\int_M\epsilon b_2 \omega^{\alpha+\gamma} \eta^2 .
\end{align*}
Picking $\gamma$ large enough such that $\epsilon b_2<b_1$, we obtain
\begin{align*}
&b_1 \int_M \omega^{\alpha+\gamma}\eta^2 \leq \int_M \epsilon(p+1)\omega^{\alpha+\gamma }|\nabla\eta|^2.
\end{align*}
Now, if $\omega\neq 0$, without loss of generality we may assume that $\omega|_{B_1}\neq 0$ for some geodesic ball $B_1$ with radius 1, it follows that there always holds true $$\int_{B_1}\omega^{\alpha+\gamma}>0.$$
By choosing $\eta\in C^{\infty}_0(B_{k+1})$, $\eta\equiv 1$ in $B_k$ and $|\nabla\eta|<4$, from the above argument we have
	\begin{align*}
		b_1\int_{B_k}\omega^{\alpha+\gamma}\leq 4\epsilon(p+1) \int_{B_{k+1}}\omega^{\alpha+\gamma}.
	\end{align*}
	By iteration on $k$, we have
	\begin{align*}
		\int_{B_k}\omega^{\alpha+\gamma}\geq \left(\frac{C}{\epsilon}\right)^k\int_{B_1}\omega^{\alpha+\gamma}.
	\end{align*}
	Since $\omega$ is uniformly bounded and $\vol(B_{k})\leq e^{(n-1)k\sqrt{\kappa}}$ by volume comparison theorem, we have
	\begin{align*}
		C_1^{\alpha+\gamma}e^{(n-1)k\sqrt{\kappa}}\geq e^{k\ln\frac{C}{\epsilon}}.
	\end{align*}
	We choose $\gamma$ such that $$\ln\frac{C}{\epsilon}>2(n-1)\sqrt{\kappa}+2$$ and then choose $k$ large enough, immediately we get a contradiction. This means $\omega\equiv 0$ and we finish the proof of this lemma.
\end{proof}
\thmref{t10} follows immediately from the above two lemmas.

Note $y_2<y_1$ and the two cases have some intersection. If we remove the intersection, we can get the following corollary.

\begin{cor}
Let $(M,g)$ be an $n$-dim($n\geq 2$) complete Riemannian manifold satisfying $\mathrm{Ric}_g\geq-(n-1)\kappa g$ for some constant $\kappa\geq0$. Assume that $v\in C^1(M)$ is an entire positive solution of \eqref{equa1} with $a>0$. Then,
\begin{enumerate}
\item if
\begin{align*}
\frac{q+r}{p-1} \leq \frac{n+1}{n-1},
\end{align*}
there holds
\begin{align*}
\frac{|\nabla v|}{v}(x)\leq \frac{(n-1)\sqrt{\kappa}}{p-1}, \quad\forall x\in M;
\end{align*}

\item if
\begin{align*}
\frac{n+1}{n-1}<\frac{q+r}{p-1}<\frac{n+3}{n-1},\quad  \forall a\in \R,
\end{align*}
there holds
\begin{align*}
\frac{|\nabla v|}{v}(x)\leq \frac{2\sqrt{\kappa}}{(p-1)\sqrt{\left(\frac{q+r}{p-1}-1\right)\left(\frac{n+3}{n-1}-\frac{q+r}{p-1}\right)}}, \quad\forall x\in M.
\end{align*}
\end{enumerate}
\end{cor}

\begin{figure}[ht]\label{fig3}
\centering
	\subfigure[$a>0$]{
		\begin{minipage}[t]{0.35\linewidth}
			\begin{tikzpicture}
				\draw[help lines, color=red!5, dashed] (-0.5,-0.5) grid (4,4);
				\path[fill=green](-0.5, 3.5)--(-0.5,-0.5)--(3.5, -0.5);
				\path[fill=yellow](-0.5, 3.5)--(-0.5,4.5)--(4.5, -0.5)--(3.5, -0.5);
				\draw[dotted]
				(-0.5, 4.5)--(0.1,3.9) node[anchor=north west]{\rotatebox{315}{$\frac{q}{p-1}+\frac{r}{p-1}=\frac{n+3}{n-1}$}} --(4.5, -0.5);
				\draw[dotted]
				(-0.5, 3.5)--(0.5,2.5) node[anchor=north west]{\rotatebox{315}{$\frac{q}{p-1}+\frac{r}{p-1}=\frac{n+1}{n-1}$} }--(3.5, -0.5);
				\filldraw[black] (0,0) circle (1pt) node[anchor=north]{$(0,0)$};
				\draw[->,ultra thick] (-0.5,0)--(4.5,0) node[right]{$r$};
				\draw[->,ultra thick] (0,-0.5)--(0,4.5) node[above]{$q$};
			\end{tikzpicture}			
		\end{minipage}
	}%
	\subfigure[$a<0$]{
		\begin{minipage}[t]{0.35\linewidth}
			\begin{tikzpicture}
				\draw[help lines, color=red!5, dashed] (-0.5,-0.5) grid (4,4);
				\path[fill=green](-0.5, 3.5)--(-0.5,4.5)--(4.5, 4.5)--(4.5, -0.5)--(3.5, -0.5);
				\path[fill=yellow](-0.5, 3.5)--(-0.5,2.5)--(2.5, -0.5)--(3.5, -0.5);
				\draw[dotted]
				(-0.5, 2.5)--(-0.2, 2.2) node[anchor=north west]{\rotatebox{315}{$\frac{q}{p-1}+\frac{r}{p-1}=1$}} --(2.5, -0.5);
				\draw[dotted]
				(-0.5, 3.5)--(0.5,2.5) node[anchor=north west]{\rotatebox{315}{$\frac{q}{p-1}+\frac{r}{p-1}=\frac{n+1}{n-1}$} }--(3.5, -0.5);
				\filldraw[black] (0,0) circle (1pt) node[anchor=north]{$(0,0)$};
				\draw[->,ultra thick] (-0.5,0)--(4.5,0) node[right]{$r$};
				\draw[->,ultra thick] (0,-0.5)--(0,4.5) node[above]{$q$};
			\end{tikzpicture}			
		\end{minipage}
	}%
	\begin{tikzpicture}			
		\path[fill=green](2, 5)--(3,5)--(3,5.2) node[anchor=west]{$|\nabla u|^2<y_1$}--(3,5.5)--(2,5.5);
		\path[fill = yellow](2, 6)--(3,6)--(3,6.2) node[anchor=west]{$|\nabla u|^2<y_2$}--(3,6.5)--(2,6.5);
		\path[fill=white](2, 3)--(3,3)--(3,3.2)  --(3,3.5)--(2,3.5);
	\end{tikzpicture}
	\centering
	\caption{ The green region is the $(q,r)$ where $\frac{|\nabla v|}{v}$ is bounded by $\frac{(n-1)\sqrt{\kappa}}{p-1}$;The yellow region  is the $(q,r)$ where $\frac{|\nabla v|}{v}$ is bounded by $$\frac{2\sqrt{\kappa}}{(p-1)\sqrt{\left(\frac{q+r}{p-1}-1\right)\left(\frac{n+3}{n-1}-\frac{q+r}{p-1}\right)}}.$$}
\end{figure}
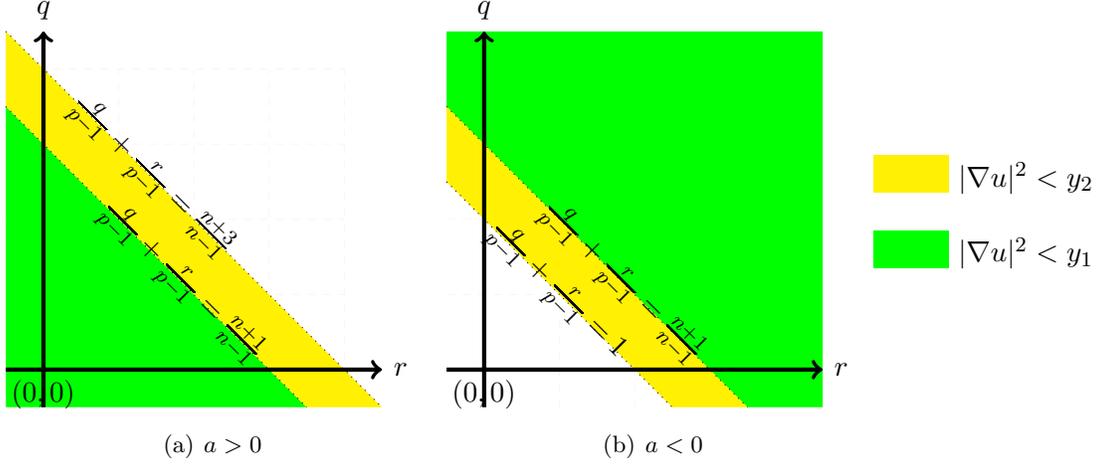

\begin{cor}
Let $(M,g)$ be an $n$-dim($n\geq 2$) complete Riemannian manifold satisfying $\mathrm{Ric}_g\geq-(n-1)\kappa g$ for some constant $\kappa\geq0$. Assume that $v\in C^1(M)$ is an entire positive solution of \eqref{equa1} with $a<0$.
	\begin{enumerate}
		\item If
		\begin{align*}
			\frac{q+r}{p-1} \geq \frac{n+1}{n-1},
		\end{align*}
		there holds
		\begin{align*}
			\frac{|\nabla v|}{v}(x)\leq \frac{(n-1)\sqrt{\kappa}}{p-1}, \quad\forall x\in M.
		\end{align*}
		\item If
		\begin{align*}
			1<\frac{q+r}{p-1}<\frac{n+1}{n-1},\quad  \forall a\in \R,
		\end{align*}
there holds
\begin{align*}
\frac{|\nabla v|}{v}(x)\leq \frac{2\sqrt{\kappa}}{(p-1)\sqrt{\left(\frac{q+r}{p-1}-1\right)\left(\frac{n+3}{n-1}-\frac{q+r}{p-1}\right)}}, \quad\forall x\in M.
\end{align*}
\end{enumerate}
\end{cor}
\medskip
	
\noindent {\it\bf{Acknowledgements}}: The author Y. Wang is supported partially by National Key Research and Development projects of China (Grant No. 2020YFA0712500).
The author J. Hu is supported by National Natural Science Foundation of China(Grant No. 12288201) and the Project of Stable Support for Youth Team in Basic Research Fields, CAS(Grant No. YSBR-001).
\medskip\medskip\medskip
	
\bibliographystyle{plain}

\end{document}